\documentclass[10pt, francais]{smfart}
\usepackage{amsfonts, amssymb, amsmath, amsthm, latexsym, enumerate, epsfig, color, mathrsfs, fancyhdr, pstricks-add}
\usepackage[all]{xy}

\usepackage[french,english]{babel}

\usepackage{mysections}

\usepackage[mac]{inputenc}

\usepackage{hyperref}

\def\N{\ensuremath{\mathbf{N}}}
\def\Z{\ensuremath{\mathbf{Z}}}
\def\Q{\ensuremath{\mathbf{Q}}}
\def\R{\ensuremath{\mathbf{R}}}
\def\C{\ensuremath{\mathbf{C}}}
\def\D{\ensuremath{\mathbf{D}}}

\def\P{\ensuremath{\mathbf{P}}}
\def\F{\ensuremath{\mathbf{F}}}

\def\E#1#2{\ensuremath{\mathbf{A}^{#1,\mathrm{an}}_{#2}}}
\def\AZ{\ensuremath{\mathbf{A}^{1,\mathrm{an}}_{\mathbf{Z}}}}
\def\AA{\ensuremath{\mathbf{A}^{1,\mathrm{an}}_{A}}}

\def\Frac{\textrm{Frac}}
\def\Spec{\textrm{Spec}}

\def\an{\textrm{an}}
\def\alg{\textrm{alg}}
\def\eme{\ensuremath{^\textrm{\`eme}}}

\def\p{\ensuremath{\mathfrak{p}}}
\def\m{\ensuremath{\mathfrak{m}}}

\def\As{\ensuremath{\mathscr{A}}}
\def\Bs{\ensuremath{\mathscr{B}}}

\def\Fs{\ensuremath{\mathscr{F}}}
\def\Gs{\ensuremath{\mathscr{G}}}
\def\Hs{\ensuremath{\mathscr{H}}}
\def\Is{\ensuremath{\mathscr{I}}}

\def\Ks{\ensuremath{\mathscr{K}}}
\def\Ls{\ensuremath{\mathscr{L}}}
\def\Ms{\ensuremath{\mathscr{M}}}
\def\Ns{\ensuremath{\mathscr{N}}}
\def\Os{\ensuremath{\mathscr{O}}}
\def\Ps{\ensuremath{\mathscr{P}}}

\def\Rs{\ensuremath{\mathscr{R}}}

\def\Hom{\ensuremath{\mathscr Hom}}

\def\Mc{\ensuremath{\mathcal{M}}}

\def\cn#1#2{\ensuremath{[\![ {#1},{#2}]\!]}}
\def\of#1#2#3{\ensuremath{\mathopen{#1}{#2}\mathclose{#3}}}

\def\la{\ensuremath{\langle}}
\def\ra{\ensuremath{\rangle}}

\def\eps{\ensuremath{\varepsilon}}

\def\b0{\ensuremath{{\boldsymbol{0}}}}

\def\l{\ensuremath{\|}}

\def\Aun{\ensuremath{A_{1^-}[\![T]\!]}}

\begin{document}
\setlength{\baselineskip}{0.55cm}	

\title{Raccord sur les espaces de Berkovich}
\author{J\'er\^ome Poineau}
\address{Institut de recherche math\'ematique avanc\'ee, 7, rue Ren\'e Descartes, 67084 Strasbourg, France}
\email{jerome.poineau@math.u-strasbg.fr}
\thanks{L'auteur est membre du projet jeunes chercheurs \og Espaces de Berkovich \fg\ de l'agence nationale de la recherche.}

\date{\today}


\begin{abstract}
\selectlanguage{french}
Nous pr\'esentons ici quelques r\'esultats autour du probl\`eme inverse de Galois. Nous commen\c{c}ons par rappeler la strat\'egie g\'eom\'etrique classique permettant de d\'emontrer que tout groupe fini est groupe de Galois sur~$\C(T)$. Nous l'appliquons dans une autre situation afin de d\'emontrer que, si~$\Ms(B)$ d\'esigne le corps des fonctions m\'eromorphes sur une partie~$B$, d'un certain type, d'un espace de Berkovich sur un corps, alors l'\'enonc\'e pr\'ec\'edent reste valable lorsque l'on remplace~$\C$ par~$\Ms(B)$. On retrouve, en particulier, le fait que tout groupe fini est groupe de Galois sur~$K(T)$, lorsque~$K$ est un corps valu\'e complet dont la valuation n'est pas triviale.


Dans un second temps, en utilisant une m\'ethode similaire, nous proposons une nouvelle preuve, purement g\'eom\'etrique, dans le langage des espaces de Berkovich sur un anneau d'entiers de corps de nombres, d'un r\'esultat de D.~Harbater assurant que tout groupe fini est groupe de Galois sur un corps de s\'eries arithm\'etiques convergentes, ainsi que quelques g\'en\'eralisations.

\vskip 0.5\baselineskip

\selectlanguage{english}
\noindent{\bf Abstract}
\vskip 0.5\baselineskip
\noindent
{\bf Patching over Berkovich Spaces.} In this text, we present a few results related to the inverse Galois problem. First we recall the geometric patching strategy that is used to handle the problem in the complex case. We use it in a different situation in order to prove that if~$\Ms(B)$ is the field of meromorphic functions over a part~$B$, satisfying some conditions, of a Berkovich space over a valued field, then every finite group is a Galois group over~$\Ms(B)(T)$. From this we derive a new proof of the fact that any finite group is a Galois group over~$K(T)$, where~$K$ is a complete valued field with non-trivial valuation.

In a second part, we deal with the following theorem by D. Harbater: every finite group is a Galois group over a field of convergent arithmetic power series. We switch to Berkovich spaces over the ring of integers of a number field and use a similar strategy to give a new and purely geometric proof of this theorem, as well as some generalizations. 
\end{abstract}

\maketitle

\pagebreak

\tableofcontents

\section*{Introduction}

Le probl\`eme inverse de Galois consiste \`a montrer que tout groupe fini peut \^etre r\'ealis\'e comme groupe de Galois sur le corps des nombres rationnels~$\Q$. La simplicit\'e de l'\'enonc\'e n'augure en rien de la difficult\'e de la question et sa r\'eponse nous \'echappe encore \`a ce jour.

Une strat\'egie due \`a D.~Hilbert consiste \`a chercher \`a r\'ealiser, tout d'abord, un groupe fini~$G$ donn\'e comme groupe de Galois sur le corps~$\Q(T)$. Ce second probl\`eme se pr\^ete \`a une approche g\'eom\'etrique. En effet, supposons que nous sachions construire un rev\^etement galoisien (ramifi\'e)~$X$ de la droite projective~$\P^1_{\Q}$ de groupe de Galois~$G$. L'extension 
$$\Ms(\P^1_{\Q}) = \Q(T) \to \Ms(X)$$
induite entre les corps de fonctions fournirait alors une solution. Le th\'eor\`eme d'ir\-r\'e\-duc\-ti\-bi\-li\-t\'e de Hilbert permet ensuite de revenir au probl\`eme initial : il assure qu'il est toujours possible de sp\'ecialiser une telle extension de fa\c{c}on \`a obtenir une extension du corps~$\Q$ dont le groupe de Galois est encore~$G$.



\bigskip

Dans ce texte, nous nous int\'eressons \`a des variantes du second probl\`eme. Nous commen\c{c}ons par rappeler, dans la section~\ref{sectioncomplexe}, une strat\'egie classique pour obtenir des rev\^etements galoisiens : construire localement des rev\^etements analytiques cycliques, puis raccorder ces rev\^etements, et enfin montrer que le rev\^etement obtenu est alg\'ebrique.


Dans la deuxi\`eme section, nous nous pla\c{c}ons dans le cadre des espaces analytiques au sens de Berkovich et appliquons la strat\'egie indiqu\'ee. Nous parvenons \`a d\'emontrer que, lorsque~$K$ est un corps ultram\'etrique complet dont la valuation n'est pas triviale, tout groupe fini est groupe de Galois sur le corps~$K(T)$. En particulier, pour tout nombre premier~$p$, tout groupe fini est groupe de Galois sur le corps~$\Q_{p}(T)$, un corps qui contient~$\Q(T)$. La d\'emonstration originale de ce r\'esultat est due \`a D.~Harbater (\emph{cf.}~\cite{arithline}, corollary~2.4) ; elle est \'ecrite dans le cadre de la g\'eom\'etrie formelle. Il existe \'egalement une preuve dans le cadre de la g\'eom\'etrie rigide, r\'edig\'ee par Q.~Liu dans~\cite{LiuGalois} en suivant une id\'ee de J.-P.~Serre. Signalons que l'absence, en g\'en\'eral, de racines primitives de l'unit\'e de tout ordre complique la premi\`ere \'etape. Aussi les deux d\'emonstrations cit\'ees font-elles appel aux constructions d\'ecrites par D.~Saltman dans l'article~\cite{Saltman}. Dans la preuve que nous proposons ici, en revanche, nous en faisons l'\'economie : un choix judicieux des lieux o\`u nous construisons les rev\^etements cycliques nous permet d'avoir recours uniquement \`a des extensions de Kummer lorsque le corps~$K$ est de caract\'eristique nulle, ou de Kummer et d'Artin-Schreier-Witt lorsqu'il est de caract\'eristique strictement positive. En toute logique, la simplification de l'arithm\'etique du probl\`eme a un co\^ut et nous utilisons un r\'esultat de g\'eom\'etrie plus compliqu\'e, mais fort naturel : un th\'eor\`eme de type GAGA relatif au-dessus d'un espace affino\"{\i}de (\emph{cf.} annexe~\ref{annexeGAGA}).

La derni\`ere section du texte est consacr\'ee \`a la construction d'extensions galoisiennes d'un sur-corps de~$\Q(T)$ d'un type diff\'erent : le corps des fractions de l'anneau~$\Z_{1^-}[\![T]\!]$ form\'e des s\'eries en une variable \`a coefficients entiers qui convergent sur le disque unit\'e ouvert complexe. Nous en d\'eduisons une nouvelle preuve du r\'esultat de D.~Harbater qui assure que tout groupe fini est groupe de Galois sur ce corps (\emph{cf.}~\cite{galoiscovers}, corollary~3.8), et l'\'etendons \`a tout corps de nombres. La d\'emonstration originale de ce r\'esultat, aboutissement de la s\'erie d'articles \cite{Harbater}, \cite{mockcovers}, \cite{algrings} et \cite{galoiscovers}, est ardue et technique ; elle bas\'ee sur des manipulations al\-g\'e\-bri\-ques des anneaux de s\'eries du m\^eme type que~$\Z_{1^-}[\![T]\!]$. Celle que nous proposons est, en revanche, purement g\'eom\'etrique. La seule difficult\'e r\'eside dans le fait que le cadre adapt\'e \`a ce probl\`eme est celui, fort naturel mais sans doute encore un peu exotique, de la droite de Berkovich sur un anneau d'entiers de corps de nombres (la construction et les propri\'et\'es de cet espace font l'objet de l'annexe~\ref{sectionberko}). Signalons, pour finir, que, contrairement \`a l'habitude, nous construisons un rev\^etement d'un ouvert d'un espace affine ; les r\'esultats de type GAGA y tombent donc en d\'efaut et nous utiliserons, pour pallier ce manque, le fait que l'ouvert en question soit un espace de Stein.


\bigskip

{\bf Notations} 


Nous d\'esignerons par~$\N$ l'ensemble des nombres entiers positifs et par~$\N^*$ le sous-ensemble form\'e de ceux qui ne sont pas nuls.

\bigskip

{\bf Remerciements} 

La derni\`ere partie de cet article a \'et\'e r\'edig\'ee au cours de l'ann\'ee que j'ai pass\'ee \`a l'universit\'e de Ratisbonne. Je souhaite remercier Klaus K\"unnemann, qui m'a permis d'y s\'ejourner, pour son accueil et  ses encouragements. Ma gratitude va \'egalement \`a Antoine Chambert-Loir dont les conseils concernant la structure de cet texte m'ont permis, je l'esp\`ere, d'en accro\^{\i}tre l'int\'er\^et et la clart\'e. Merci \'egalement \`a Antoine Ducros de m'avoir communiqu\'e ses notes sur les th\'eor\`emes GAGA.

\section{Strat\'egie de raccord}\label{sectioncomplexe}

Nous rappelons ici une d\'emonstration classique du fait que tout groupe fini est groupe de Galois d'un rev\^etement de la vari\'et\'e alg\'ebrique~$\P^1_{\C}$. Ce n'est qu'un pr\'etexte pour pr\'esenter la strat\'egie de raccord que nous utiliserons constamment par la suite.


Consid\'erons tout d'abord le cas des groupes cycliques. Pour disposer de plus de souplesse, nous allons commencer par construire des rev\^etements de la vari\'et\'e analytique~$\P^1(\C)$, et m\^eme des rev\^etements de petits ouverts de cette vari\'et\'e. Soit $m\in\N^*$. Choisissons un point~$P$ de~$\P^1(\C)$, une coordonn\'ee locale~$z$ au voisinage de ce point, un disque ouvert~$D_{P}$ sur lequel elle est d\'efinie et un disque ferm\'e~$E_{P}$ de rayon strictement positif contenu dans~$D_{P}$. Soient~$a$ et~$b$ deux points distincts de~$E_{P}$. Consid\'erons le rev\^etement connexe et lisse~$X_{P}$ du disque~$D_{P}$ donn\'e par l'\'equation
$$u^m = (z-a)(z-b)^{m-1}.$$
C'est un rev\^etement galoisien de groupe $\Z/m\Z$. En outre, il est trivial au-dessus du compl\'ementaire du disque~$E_{P}$. Remarquons que pour d\'eterminer le groupe de Galois nous avons utilis\'e le fait que le corps~$\C$ contienne une racine primitive $m\eme$ de l'unit\'e. 

Nous allons maintenant recoller des rev\^etements du type pr\'ec\'edent afin d'en construire qui poss\`edent des groupes de Galois finis arbitraires. Fixons un groupe fini~$G$. Notons~$n$ son ordre et choisissons-en des g\'en\'erateurs $g_{1},\ldots,g_{t}$, avec $t\in\N^*$. Soit $i\in\cn{1}{t}$. Notons~$n_{i}$ l'ordre de l'\'el\'ement~$g_{i}$ dans le groupe~$G$ et posons $d_{i}=n/n_{i}$. Choisissons un point~$P_{i}$ de~$\P^1(\C)$ et construisons, par la m\'ethode du paragraphe pr\'ec\'edent, un $\Z/n_{i}\Z$-rev\^etement~$X_{P_{i}}$ au-dessus d'un disque ouvert~$D_{P_{i}}$ et trivial hors d'un disque ferm\'e $E_{P_{i}} \subset D_{P_{i}}$. Indexons les feuillets de ce rev\^etement par les entiers compris entre~$0$ et~$n_{i}-1$ de fa\c{c}on compatible avec l'action du groupe~$\Z/n_{i}\Z\simeq \la g_{i} \ra$. Consid\'erons maintenant~$\textrm{Ind}_{\la g_{i} \ra}^G(X_{P_{i}})$, le $G$-rev\^etement induit par le $\la g_{i} \ra$-rev\^etement~$X_{P_{i}}$. Rappelons qu'il est constitu\'e topologiquement de~$d_{i}$ copies de~$X_{P_{i}}$. Nous pouvons envoyer, de fa\c{c}on bijective, les feuillets de ce rev\^etement sur les \'el\'ements du groupe. Pour ce faire, choisissons, dans~$G$, des repr\'esentants $a_{i,0},\ldots,a_{i,d_{i}-1}$ des \'el\'ements du quotient $G/\la g_{i} \ra$. L'application qui envoie le feuillet index\'e par~$k$ de la copie index\'ee par~$l$ de~$X_{P_{i}}$ sur l'\'el\'ement~$a_{i,l}g_{i}^k$ de~$G$ est bijective. Nous pouvons alors d\'ecrire l'action du groupe~$G$ sur le rev\^etement~$\textrm{Ind}_{\la g_{i} \ra}^G(X_{P_{i}})$ de la fa\c{c}on suivante : l'\'el\'ement~$g$ de~$G$ envoie le feuillet associ\'e \`a l'\'el\'ement~$h$ de~$G$ sur le feuillet associ\'e \`a l'\'el\'ement~$hg$.

Notons~$D'$ le compl\'ementaire de la r\'eunion des disques $E_{P_{1}},\ldots,E_{P_{t}}$ dans~$\P^1(\C)$. Con\-si\-d\'e\-rons le $G$-rev\^etement $\textrm{Ind}_{\la e \ra}^G(D')$ induit par le rev\^etement trivial de~$D'$ et indexons ses feuillets par les \'el\'ements de~$G$ de fa\c{c}on compatible avec l'action de ce groupe.

Raccordons, maintenant, les rev\^etements que nous venons de construire. Nous supposerons que les disques~$D_{P_{i}}$, avec $i\in\cn{1}{t}$, sont deux \`a deux disjoints. Nous pouvons facilement nous ramener \`a ce cas en les r\'eduisant, si besoin est. Pour tout \'el\'ement~$i$ de~$\cn{1}{t}$, nous recollons alors,  au-dessus de l'intersection~$D_{P_{i}}\cap D'$, les feuillets associ\'es aux m\^emes \'el\'ements du groupe~$G$ (\emph{cf.} figure \addtocounter{figure}{1}\thefigure\addtocounter{figure}{-1}). Nous obtenons ainsi un rev\^etement~$Y$ de~$\P^1(\C)$ dont on v\'erifie facilement qu'il est connexe, lisse et galoisien de groupe~$G$. 

\begin{figure}[htb]\label{figurerecollement}
\begin{center}
\newrgbcolor{wwwwww}{0.4 0.4 0.4}
\psset{xunit=0.6cm,yunit=0.6cm,algebraic=true,dotstyle=o,dotsize=3pt 0,linewidth=0.8pt,arrowsize=3pt 2,arrowinset=0.25}
\begin{pspicture*}(-4,-9)(19,7)
\psline(-3,-7)(-1,-7)
\psline[linestyle=dashed,dash=2pt 2pt](-1,-7)(0,-7)
\psline(0,-7)(2,-7)
\psline(13,-7)(15,-7)
\psline[linestyle=dashed,dash=2pt 2pt](15,-7)(16,-7)
\psline(16,-7)(18,-7)
\psline(-3,-8)(-1,-8)
\psline(0,-8)(15,-8)
\psline(16,-8)(18,-8)
\psline[linestyle=dashed,dash=2pt 2pt,linecolor=wwwwww](15,6)(15,3)
\psline[linestyle=dashed,dash=2pt 2pt,linecolor=wwwwww](15,3)(16,3)
\psline[linestyle=dashed,dash=2pt 2pt,linecolor=wwwwww](16,3)(16,6)
\psline[linestyle=dashed,dash=2pt 2pt,linecolor=wwwwww](16,6)(15,6)
\psline[linestyle=dashed,dash=2pt 2pt,linecolor=wwwwww](15,1)(15,-2)
\psline[linestyle=dashed,dash=2pt 2pt,linecolor=wwwwww](15,-2)(16,-2)
\psline[linestyle=dashed,dash=2pt 2pt,linecolor=wwwwww](16,-2)(16,1)
\psline[linestyle=dashed,dash=2pt 2pt,linecolor=wwwwww](16,1)(15,1)
\psline(15,5)(13,5)
\psline(15,4.5)(13,4.5)
\psline(15,5.5)(13,5.5)
\psline(15,0)(13,0)
\psline(15,0.5)(13,0.5)
\psline[linestyle=dashed,dash=2pt 2pt,linecolor=wwwwww](-1,4)(-1,0)
\psline[linestyle=dashed,dash=2pt 2pt,linecolor=wwwwww](-1,0)(0,0)
\psline[linestyle=dashed,dash=2pt 2pt,linecolor=wwwwww](0,0)(0,4)
\psline[linestyle=dashed,dash=2pt 2pt,linecolor=wwwwww](0,4)(-1,4)
\psline(0,2.5)(2,2.5)
\psline(0,3.5)(2,3.5)
\psline(-1,3.5)(-3,3.5)
\psline(-1,2.5)(-3,2.5)
\psline(16,5.5)(18,5.5)
\psline(16,5)(18,5)
\psline(16,4.5)(18,4.5)
\psline(16,0.5)(18,0.5)
\psline(16,0)(18,0)
\parametricplot{1.2732710714015814}{2.2280285817731684}{1*12.17*cos(t)+0*12.17*sin(t)+9.43|0*12.17*cos(t)+1*12.17*sin(t)+-7.13}
\psline(0,1.5)(2,1.5)
\psline(0,0.5)(2,0.5)
\psline(15,-1.5)(13,-1.5)
\parametricplot{3.9985669820165612}{4.893706880450967}{1*13.17*cos(t)+0*13.17*sin(t)+10.62|0*13.17*cos(t)+1*13.17*sin(t)+11.46}
\psline(-1,1.5)(-3,1.5)
\psline(-1,0.5)(-3,0.5)
\psline(16,-1.5)(18,-1.5)
\rput[tl](1,3.6){$\vdots$}
\rput[tl](1,2.6){$\vdots$}
\rput[tl](1,1.6){$\vdots$}
\rput[tl](-2,3.6){$\vdots$}
\rput[tl](-2,2.6){$\vdots$}
\rput[tl](-2,1.6){$\vdots$}
\rput[tl](14,4.3){$\vdots$}
\rput[tl](17,4.3){$\vdots$}
\rput[tl](14,-0.1){$\vdots$}
\rput[tl](17,-0.1){$\vdots$}
\rput[tl](15.5,2.87){$\vdots$}
\rput[tl](-0.5,5.85){$\vdots$}
\rput[tl](-0.5,-0.69){$\vdots$}
\rput[tl](15.5,2.19){$\vdots$}
\rput[tl](-0.5,5.17){$\vdots$}
\rput[tl](-0.5,0){$\vdots$}
\rput[tl](7.7,-2.82){$Y$}
\rput[tl](7.3,-5.5){$\mathbf{P}^1(\mathbf{C})$}
\psline{->}(7.96,-3.6)(7.98,-5.2)
\rput[tl](14.8,-0.5){$d_j-1$}
\rput[tl](15.41,5.2){$1$}
\rput[tl](7.79,-7.17){$D'$}
\rput[tl](-0.9,-6.3){$E_{P_{i}}$}
\rput[tl](-2.93,-6.3){$D_{P_{i}}$}
\rput[tl](15.1,-6.3){$E_{P_{j}}$}
\rput[tl](16.9,-6.3){$D_{P_{j}}$}
\rput[tl](-0.63,3){$l$}
\rput[tl](2.39,3.81){$0$}
\rput[tl](2.5,2.7){$k$}
\rput[tl](2.5,1.76){$k'$}
\rput[tl](2.5,0.74){$n_i-1$}
\rput[tl](12.2,5.75){$0$}
\rput[tl](12.2,5.27){$1$}
\rput[tl](12.2,4.76){$2$}
\rput[tl](12.2,0.77){$0$}
\rput[tl](12.2,0.26){$1$}
\rput[tl](12.2,-1.78){$n_j-1$}
\rput[tl](6.77,4.59){$a_{i,l}g_i^k=a_{j,1}g_j^2$}
\rput[tl](5.5,0.1){$a_{i,l}g_i^{k'}=a_{j,d_j-1}g_j^{-1}$}
\end{pspicture*}
\caption{Raccord de rev\^etements cycliques.}
\end{center}
\end{figure}
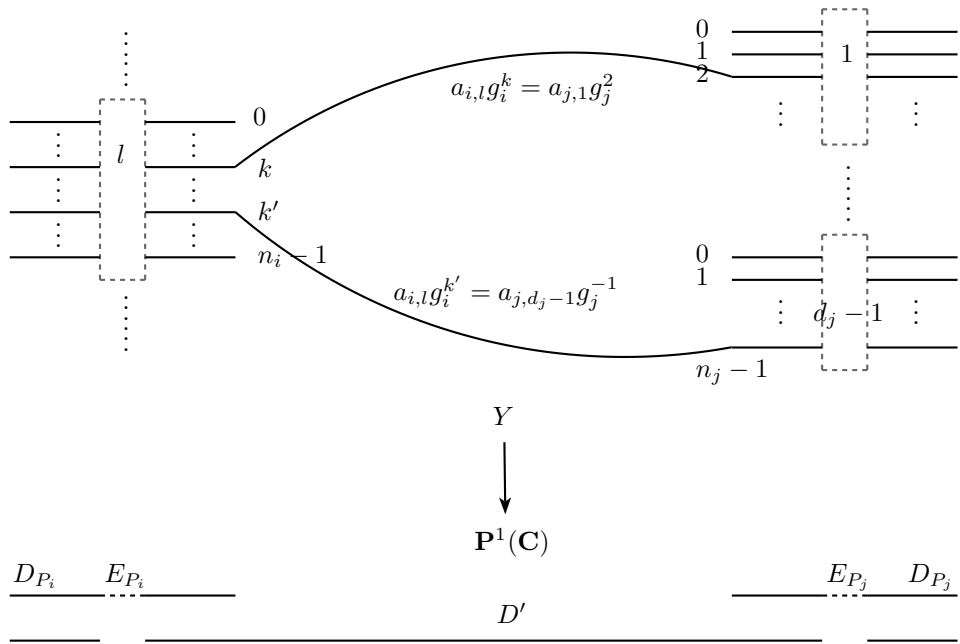

Ainsi avons-nous obtenu une vari\'et\'e analytique complexe~$Y$ v\'erifiant les propri\'et\'es requises. Il nous reste \`a montrer que c'est, en r\'ealit\'e, une vari\'et\'e alg\'ebrique. Ce r\'esultat d\'ecoule du th\'eor\`eme d'existence de Riemann ou, si l'on veut, des th\'eor\`emes GAGA de J.-P.~Serre. Nous avons  finalement obtenu le r\'esultat suivant :

\begin{thm}\label{surC}
Tout groupe fini est groupe de Galois d'une extension du corps~$\C(T)$.
\end{thm}

\bigskip

Pour r\'esumer, rappelons en quelques mots la strat\'egie de la preuve :
\begin{enumerate}
\item Construire des rev\^etements cycliques sur de petits ouverts, triviaux au voisinage du bord.
\item Raccorder ces rev\^etements.
\item Montrer que le rev\^etement obtenu est alg\'ebrique.
\end{enumerate}

D.~Harbater l'a d\'evelopp\'ee dans plusieurs contextes et utilis\'ee pour d\'emontrer de nombreux r\'esultats. Nous renvoyons le lecteur d\'esireux d'en savoir plus au texte~\cite{patching}.

La simplicit\'e de cette strat\'egie de raccord (\og patching \fg\ chez D.~Harbater) invite \`a l'appliquer dans de nombreux contextes g\'eom\'etriques, pour peu que l'on dispose d'une bonne notion de \og petits ouverts \fg\ et de th\'eor\`emes d'alg\'ebricit\'e. Ce n'est pas le cas de la g\'eom\'etrie alg\'ebrique, o\`u deux ouverts non vides de la droite projective se coupent toujours, mais ce devrait l'\^etre pour toute g\'eom\'etrie analytique. Dans la suite de ce texte, nous illustrerons cette id\'ee en appliquant la strat\'egie indiqu\'ee dans le cadre des espaces de Berkovich sur un corps ultram\'etrique complet, \`a la section~\ref{sectiondemoberko}, puis sur un anneau d'entiers de corps de nombres, \`a la section~\ref{sectiondemo}.

\section{Probl\`eme inverse de Galois sur une droite relative} \label{sectiondemoberko}

Soient~$k$ un corps muni d'une valeur absolue ultram\'etrique pour laquelle il est complet et~$X$ un espace $k$-analytique g\'eom\'etriquement irr\'eductible. Nous noterons~$\Os$ le faisceau structural sur cet espace. Soit~$B$ une partie de l'espace~$X$. 

Notons~$Y$ le produit fibr\'e $X\times_{k}\P^{1,\an}_{k}$, $\pi : Y \to X$ et $\lambda : Y \to \P^{1,\an}_{k}$ les morphismes naturels de projection. Signalons que, d'apr\`es~\cite{excellence}, th\'eor\`emes~7.16 et~8.4, l'espace~$Y$ est g\'eom\'etriquement irr\'eductible. Notons~$Y(B)$ l'image r\'eciproque de~$B$ par le morphisme~$\pi$. 


Dans les num\'eros~\ref{constructionlocaleberko}, \ref{numerorecollementberko} et~\ref{conclusion}, nous supposerons que la partie~$B$ est compacte (pour les applications, dans la preuve des corollaires~\ref{corcomplet} et~\ref{corfertile}, elle sera m\^eme r\'eduite \`a un point). Nous la munirons alors du faisceau des fonctions surconvergentes, c'est-\`a-dire du faisceau~$j^{-1}\Os$, o\`u $j : B \hookrightarrow X$ d\'esigne l'inclusion. De fa\c{c}on g\'en\'erale, nous utiliserons, sans plus le pr\'eciser, le faisceau des fonctions surconvergentes pour toute partie compacte. En particulier, l'espace localement annel\'e associ\'e \`a l'espace~$Y(B)$ que nous obtiendrons ainsi ne sera autre que le germe $(Y,\pi^{-1}(B))$ au sens de V.~Berkovich (\emph{cf.}~\cite{bleu}, \S 3.4 ou~\cite{vanishingcyclesanalytic}, \S 2).

\subsection{Construction locale de rev\^etements cycliques}\label{constructionlocaleberko}

Dans le cas complexe, la construction locale \'etait particuli\`erement simple, car nous disposions de racines primitives de l'unit\'e de tout ordre. Lorsque nous cherchons \`a construire un rev\^etement cyclique dont l'ordre~$n$ est premier \`a l'exposant caract\'eristique du corps de base, la situation n'est gu\`ere plus compliqu\'ee, car nous disposons encore de racines primitives $n^\textrm{\`emes}$ de l'unit\'e sur certains ouverts. Nous utiliserons alors une extension de Kummer bien choisie. Dans les autres cas, nous ferons appel \`a la th\'eorie d'Artin-Schreier-Witt. Nous noterons~$p$ l'exposant caract\'eristique du corps~$k$. 

Fixons une extension finie et s\'eparable~$K$ de~$k$. Soit $P\in k[T]$ le polyn\^ome minimal unitaire d'un \'el\'ement primitif de cette extension. Notons~$t$ le point de~$\P^{1,\an}_{k}$ d\'efini par l'annulation de ce polyn\^ome.

\subsubsection{Extensions de Kummer}\label{Kummer}

Soit~$n$ un entier sup\'erieur ou \'egal \`a~$2$ et premier \`a~$p$. Supposons que le corps~$K$ contient une racine primitive~$n\eme$ de l'unit\'e. 

Puisque l'anneau local au point~$t$ est hens\'elien, il contient une racine primitive $n\eme$ de l'unit\'e, que nous noterons~$\zeta$. Cette racine est d\'efinie sur un voisinage du point~$t$, que nous pouvons supposer de la forme
$$V_{u} = \left\{\left.y\in\P^{1,\an}_{k}\,\right|\, |P(y)|\le u\right\},$$
avec $u>0$.

Posons $V = \lambda^{-1}(V_{u})$. Nous noterons encore~$P$ et~$\zeta$ les r\'etrotirettes des \'el\'ements~$P$ et~$\zeta$ par le morphisme~$\lambda$. Remarquons que l'\'el\'ement~$\zeta$ de~$\Os(V)$ est encore une racine primitive~$n\eme$ de l'unit\'e.


Soient~$U$ un voisinage de~$B$ dans~$X$ et~$\alpha$ un \'el\'ement de~$\Os(U)$ nul en tout point de~$B$. Posons
$$Q(S) = S^n-P^n-\alpha \in \Os(\pi^{-1}(U))[S].$$
D\'efinissons un pr\'e\-fais\-ceau~$\Fs$ sur~$\pi^{-1}(U)$ en posant, pour toute partie ouverte~$W$ de~$\pi^{-1}(U)$,
$$\Fs(W) = \Os(W)[S]/(Q(S))$$
et en utilisant les morphismes de restriction induits par ceux du faisceau~$\Os$. Le caract\`ere unitaire du polyn\^ome~$Q$ assure que~$\Fs$ est un faisceau de $\Os_{\pi^{-1}(U)}$-alg\`ebres coh\'erent.

\begin{rem}
Le faisceau~$\Fs$ est l'image directe du faisceau structural d'une courbe analytique sur~$Y$. Celle-ci nous est donn\'ee comme un rev\^etement ramifi\'e de degr\'e~$n$ de~$\pi^{-1}(U)$.
\end{rem}

Soit $v\in\of{]}{0,u}{]}$. Posons 
$$V(B) = \left\{y\in Y(B)\, \big|\, |P(y)|\le v\right\}$$
et
$$V'(B) = \{y\in V(B)\, |\, P(y)\ne 0\}.$$
Remarquons que le compl\'ementaire de~$V'(B)$ dans~$V(B)$ est ferm\'e dans~$Y(B)$.

\begin{prop}\label{nonramifieberko}
Il existe un isomorphisme de $\Os_{V'(B)}$-alg\`ebres
$$\varphi : \Fs \to \Os^n$$
tel que, pour tout ouvert~$W$ de~$V'(B)$ et tout \'el\'ement~$s$ de~$\Fs(W)$, nous ayons
$$\varphi(\zeta s) = \tau(\varphi(s)),$$
o\`u~$\tau$ d\'esigne l'automorphisme du faisceau~$\Os^n$ qui consiste \`a faire agir la permutation cyclique $(1\ 2\ \cdots\ n)$ sur les coordonn\'ees.
\end{prop}
\begin{proof}
Le rayon de convergence de la s\'erie
$$(1+T)^{1/n} = \sum_{i=0}^{+\infty} C_{1/n}^i\, T^i \in k[\![T]\!]$$
est strictement positif. La s\'erie  
$$P\, \sum_{i=0}^{+\infty} C_{1/n}^i\, \left(\frac{\alpha}{P^n}\right)^{i}$$
d\'efinit donc un \'el\'ement~$\omega$ de~$\Os(V'(B))$, qui est une racine du polyn\^ome~$Q$. Nous avons alors l'\'egalit\'e
$$Q(S) = S^n-P^n-\alpha = \prod_{j=0}^n (S - \zeta^{j} \omega) \textrm{ dans } \Os(V'(B)).$$
Par cons\'equent, le morphisme
$$\begin{array}{ccc}
\Fs & \to & \Os^n\\
R(S) & \mapsto & \left(R(\omega), R(\zeta^{-1} \omega),\ldots, R(\zeta^{-(n-1)} \omega)\right) 
\end{array}$$
est un isomorphisme. On v\'erifie imm\'ediatement qu'il satisfait la condition requise.
\end{proof}

\begin{rem}
La premi\`ere partie du r\'esultat signifie que le rev\^etement associ\'e au faisceau~$\Fs$ est trivial au-dessus de~$V'(B)$. La seconde assure que le groupe $\of{\la}{\zeta}{\ra} \simeq \Z/n\Z$ agit sur le rev\^etement par une permutation cyclique des feuillets du lieu trivial.  
\end{rem}


\bigskip

Nous allons maintenant imposer des conditions sur les donn\'ees~$B$ et~$\alpha$ de fa\c{c}on que le faisceau~$\Fs$ soit associ\'e \`a un rev\^etement irr\'eductible.

\begin{defi}
Nous dirons que la partie~$B$ de~$X$ satisfait la {\bf condition~(CGI)} si elle est compacte et poss\`ede un syst\`eme fondamental de voisinages affino\"{\i}des g\'eo\-m\'e\-tri\-que\-ment int\`egres.
\end{defi}

Sous les conditions de cette d\'efinition, l'anneau~$\Os(B)$ est int\`egre et la partie~$B$ connexe. En particulier, le principe du prolongement analytique vaut sur~$B$ et les anneaux locaux en tous les points de~$B$ sont int\`egres.

\bigskip




\begin{lem}\label{Vxirr}
Supposons que la partie~$B$ de~$X$ satisfait la condition~(CGI). Alors, la partie~$V(B)$ de~$Y$ poss\`ede un syst\`eme fondamental de voisinages affino\"ides irr\'eductibles.

En particulier, pour tout point~$z$ de~$V(B)$, le morphisme naturel
$$\Os(V(B)) \to \Os_{z}$$
est injectif.
\end{lem}
\begin{proof}
Soient~$U'$ un voisinage affino\"ide g\'eom\'etriquement irr\'eductible de~$B$ dans~$U$ et~$v'$ un nombre r\'eel strictement sup\'erieur \`a~$v$. La partie de~$Y$ d\'efinie par
$$\{z\in \pi^{-1}(U')\, \big|\, |P(z)|\le v'\}$$
est alors irr\'eductible, d'apr\`es~\cite{excellence}, th\'eor\`eme~8.4. En effet, ce n'est autre que le produit, au-dessus de~$k$, de l'espace g\'eom\'etriquement irr\'eductible~$U'$ par l'espace
$$\left\{\left. x\in\P^{1,\an}_{k}\, \right|\, |P(x)|\le v'\right\},$$
qui est irr\'eductible, car le polyn\^ome~$P$ est irr\'eductible sur~$k$.

La condition~(CGI) assurent que l'ensemble des parties de la forme pr\'ec\'edente est un syst\`eme fondamental de voisinages de~$V(B)$ dans~$Y$. 
\end{proof}

{\bf Nous supposerons d\'esormais que la partie compacte~$\boldsymbol{B}$ de~$\boldsymbol{X}$ satisfait la condition~(CGI).}

\bigskip

\begin{defi}
Nous dirons que l'\'el\'ement~$\alpha$ de~$\Os(B)$ satisfait la condition~($\textrm{I}_{n,K}$) s'il existe un point~$x$ de~$B$ qui v\'erifie les deux conditions suivantes :
\begin{enumerate}[\it i)]
\item les corps~$K$ et~$\Frac(\Os_{x})$ sont lin\'eairement disjoints sur~$k$ ;
\item le polyn\^ome~$S^n-\alpha$ est irr\'eductible sur le corps $\Frac(\Os_{x})\otimes_{k} K$. 
\end{enumerate}
\end{defi}

\begin{lem}\label{lemintlocalberko}
Supposons que l'\'el\'ement~$\alpha$ de~$\Os(B)$ satisfait la condition~($\textrm{I}_{n,K}$). Alors, le polyn\^ome~$Q(S)=S^n-P^n-\alpha$ est irr\'eductible sur le corps Frac$(\Os(V(B)))$. En particulier, l'anneau $\Fs(V(B))$ est int\`egre.
\end{lem}
\begin{proof} 
Soit~$x$ un point de~$B$ satisfaisant les propri\'et\'es \'enonc\'ees dans la d\'efinition de la condition~($\textrm{I}_{n,K}$). Puisque les corps~$K$ et~$\Frac(\Os_{x})$ sont lin\'eairement disjoints sur~$k$, le polyn\^ome~$P$ est ir\-r\'e\-duc\-ti\-ble sur~$\Os_{x}$. Il l'est donc encore dans $\Hs(x)[T]$, puisque le corps~$\kappa(x)$ et l'anneau local~$\Os_{x}$ sont hens\'eliens. 

Notons~$Z(B)$ l'ensemble des points de~$Y(B)$ en lesquels la fonction~$P$ est nulle. D'apr\`es le raisonnement pr\'ec\'edent, la trace de la fibre~$\pi^{-1}(x)$ sur~$Z$ comporte un seul point, que nous noterons~$z$. Notons~$\Os_{Z}$ le faisceau structural sur~$Z$. Nous avons alors un isomorphisme
$$\Os_{x}[T]/(P(T)) \xrightarrow[]{\sim} \Os_{Z,z}.$$

Le polyn\^ome~$S^n-\alpha$ est donc irr\'eductible sur le corps $\Frac(\Os_{Z,z})$, isomorphe \`a $\Frac(\Os_{x})\otimes_{k} K$. Par cons\'equent, le polyn\^ome~$Q(S)$ est irr\'eductible sur le corps $\Frac(\Os_{z})$.

D'apr\`es le lemme~\ref{Vxirr}, le morphisme naturel $\Os(V(B))\to \Os_{z}$ est injectif. Par cons\'equent, le corps $\Frac(\Os(V(B)))$ est un sous-corps de $\Frac(\Os_{z})$. On en d\'eduit que le polyn\^ome~$Q(S)$ est irr\'eductible sur le corps $\Frac(\Os(V(B)))$.

Le polyn\^ome~$Q(S)$ \'etant unitaire, l'unicit\'e de la division euclidienne assure que le morphisme 
$$\Os(V(B))[S]/(Q(S)) \to \Frac(\Os(V(B)))[S]/(Q(S))$$
est injectif. Puisque l'anneau au but est int\`egre, celui \`a la source, qui n'est autre que l'anneau~$\Fs(V(B))$, l'est \'egalement.
\end{proof}

\begin{rem}
Ce r\'esultat signifie que la courbe associ\'ee au faisceau~$\Fs$ est int\`egre, c'est-\`a-dire r\'eduite et irr\'eductible.
\end{rem}

\subsubsection{Extensions d'Artin-Schreier-Witt}\label{Artin-Schreier-Witt}

Il nous reste \`a traiter le cas des groupes cycliques dont l'ordre n'est pas premier \`a l'exposant caract\'eristique~$p$ du corps~$k$. Nous supposons d\'esormais que~$p$ est un nombre premier et chercherons \`a construire un rev\^etement cyclique d'ordre~$n=p^r$, o\`u~$r$ est un entier sup\'erieur \`a~$1$. Il s'agit essentiellement de remplacer, dans le num\'ero pr\'ec\'edent, la th\'eorie de Kummer par celle d'Artin-Schreier-Witt. Nous nous contenterons d'indiquer les grandes lignes de la preuve.

\bigskip

Dans ce num\'ero, comme dans le pr\'ec\'edent, {\bf nous supposerons que la partie compacte~$\boldsymbol{B}$ de~$\boldsymbol{X}$ satisfait la condition~(CGI).}

Soit~$\alpha$ un \'el\'ement de~$\Os(B)$ nul en tout point de~$B$. 

\begin{defi}
Nous dirons que l'\'el\'ement~$\alpha$ de~$\Os(B)$ satisfait la condition~($\textrm{I}_{p,K}$) s'il existe un point~$x$ de~$B$ qui v\'erifie les deux conditions suivantes :
\begin{enumerate}[\it i)]
\item les corps~$K$ et~$\Frac(\Os_{x})$ sont lin\'eairement disjoints sur~$k$ ;
\item le polyn\^ome~$S^p-\alpha$ est irr\'eductible sur le corps $\Frac(\Os_{x})\otimes_{k} K$. 
\end{enumerate}
\end{defi}

Soit~$v>0$. Posons
$$V(B) = \left\{y\in Y(B)\, \big|\, |P(y)|\le v\right\}$$
et
$$V'(B) = \{y\in V(B)\, |\, P(y)\ne 0\}.$$

Notons~$W_{r}$ l'anneau des vecteurs de Witt de longueur~$r$ sur~$\Os(V(B))[S_{0},\ldots,S_{r-1}]$. Posons
$$S=(S_{0},\ldots,S_{r-1}) \in W_{r}$$  
et, pour tout \'el\'ement~$a$ de~$\Os(V(B))[S_{0},\ldots,S_{r-1}]$,
$$\{a\}=(a,0,\ldots,0)\in W_{r}.$$
Pour tout $i\in\cn{0}{r-1}$, d\'efinissons un polyn\^ome $Q_{i}(S_{0},\ldots,S_{r-1})$ \`a coefficients dans~$\Os(V(B))$ par la formule
$$(Q_{0},\ldots,Q_{r-1}) = F(S) - \{P\}^{p-1}S - \{\alpha\} \textrm{ dans } W_{r}.$$
Le groupe~$\Z/p^r\Z$ agit sur l'anneau $\Os(V(B))[S_{0},\ldots,S_{r-1}]/(Q_{0},\ldots,Q_{r-1})$ en laissant stable~$\Os(V(B))$ et en envoyant~$S_{i}$, pour $i\in\cn{0}{r-1}$, sur la $(i+1)^\textrm{\`eme}$ coordonn\'ee du vecteur~$S+\{P\}$ dans~$W_{r}$. Par analogie avec la construction pr\'ec\'edente, nous d\'efinissons un faisceau sur~$V(B)$ par
$$\Fs = \Os[S_{0},\ldots,S_{r-1}]/(Q_{0},\ldots,Q_{r-1}).$$ 

Les propri\'et\'es des vecteurs de Witt montrent que, pour tout $i\in\cn{0}{r-1}$, nous avons
$$Q_{i} = S_{i}^p - P^{(p-1)p^i}S_{i} \mod (\alpha, Q_{0},\ldots,Q_{i-1}).$$

Soit $y\in V'(B)$. L'image~$S_{0}^p-P^{p-1}S_{0}$ du polyn\^ome~$Q_{0}(S_{0})$ dans le corps r\'esiduel~$\kappa(y)$ de l'anneau local~$\Os_{y}$ est scind\'e \`a racines simples. Puisque cet anneau local est hens\'elien, le polyn\^ome~$P_{0}(S_{0})$ poss\`ede~$p$ racines simples dans~$\Os_{y}$. En raisonnant par r\'ecurrence sur le nombre de variables, on montre ainsi que le syst\`eme d'\'equa\-tions polynomiales donn\'e par $Q_{0},\ldots,Q_{r-1}$ poss\`ede exactement~$p^r$ racines $\omega_{1},\ldots,\omega_{p^r}$ dans~$\Os_{y}^r$ et que le morphisme
$$\begin{array}{cccc}
\psi : & \Os_{y}[S_{0},\ldots,S_{r-1}]/(Q_{0},\ldots,Q_{r-1}) & \to & \Os_{y}^{p^r}\\
& R(S_{0},\ldots,S_{r-1}) & \mapsto & (R(\omega_{i}))_{1\le i\le p^r}
\end{array}$$
est un isomorphisme. On en d\'eduit que le rev\^etement est trivial sur~$V'(B)$, comme \`a la proposition~\ref{nonramifieberko}.

Supposons que l'\'el\'ement~$\alpha$ de~$\Os(B)$ satisfait la condition~($\textrm{I}_{p,K}$). L'\'enonc\'e du lemme~\ref{lemintlocalberko} vaut alors encore. Pour le d\'emontrer, l'on remplace simplement les arguments de la th\'eorie de Kummer par ceux de la th\'eorie d'Artin-Schreier-Witt. L'argument-cl\'e consiste \`a utiliser le fait que le polyn\^ome $S^p-\alpha$ est irr\'eductible sur~$\Os_{x}$, o\`u~$x$ est un point de~$B$ satisfaisant les propri\'et\'es \'enonc\'ees dans la d\'efinition de la condition~($\textrm{I}_{p,K}$), et \`a en d\'eduire que le polyn\^ome $S^p-P^{p-1}S - \alpha$ est irr\'eductible sur~$\Os_{z}$, o\`u~$z$ d\'esigne l'unique point de la fibre~$\pi^{-1}(x)$ en lequel~$P$ s'annule.

\subsection{Raccord et retour \`a l'alg\`ebre}\label{numerorecollementberko}

Soit~$G$ un groupe fini. Soient $g_{1},\ldots,g_{t}$, avec~$t\in\N^*$, des g\'en\'erateurs du groupe~$G$. Nous pouvons les choisir de fa\c{c}on que, pour tout \'el\'ement~$i$ de~$\cn{1}{t}$, il existe un nombre premier~$p_{i}$ et un entier~$r_{i}\ge 1$ tels que le sous-groupe de~$G$ engendr\'e par~$g_{i}$ soit cyclique d'ordre~$p_{i}^{r_{i}}$. Nous supposerons qu'il existe $s\in\cn{1}{t}$ tel que, pour tout $i\le s$, $p_{i}\ne p$ et, pour tout $i\ge s+1$, $p_{i}=p$.

Nous nous pla\c{c}ons sous les hypoth\`eses suivantes :
\begin{itemize}
\item la partie~$B$ de~$X$ satisfait la condition~(CGI) ;
\item pour tout \'el\'ement~$i$ de~$\cn{1}{s}$, il existe une extension s\'eparable~$K_{i}$ de~$k$ contenant une racine primitive $(p_{i}^{r_{i}})\eme$ de l'unit\'e et un \'el\'ement~$\alpha_{i}$ de~$\Os(B)$ qui satisfait la condition~$\textrm{I}_{p_{i}^{r_{i}},K_{i}}$ ;
\item pour tout \'el\'ement~$i$ de~$\cn{s+1}{t}$, il existe une extension s\'eparable~$K_{i}$ de~$k$ et un \'el\'ement~$\alpha_{i}$ de~$\Os(B)$ qui satisfait la condition~$\textrm{I}_{p,K_{i}}$ ;
\item les corps $K_{1},\ldots,K_{t}$ sont deux \`a deux non isomorphes.
\end{itemize}

Supposons que les \'el\'ement $\alpha_{1},\ldots,\alpha_{t}$ sont nuls en tout point de~$B$. Soit~$i$ un \'el\'ement de~$\cn{1}{t}$. Construisons par la m\'ethode du num\'ero~\ref{Kummer} ou~\ref{Artin-Schreier-Witt} un rev\^etement galoisien de groupe~$\Z/p_{i}^{r_{i}}\Z$. Il est d\'efini au-dessus d'une partie~$V_{i}$ et trivial au-dessus d'une partie~$V'_{i}$. Notons $\textrm{Ind}_{\la g_{i} \ra}^G(V_{i})$ le $G$-rev\^etement induit. 

Puisque les corps $K_{1},\ldots,K_{t}$ sont deux \`a deux non isomorphes, nous pouvons choisir les parties $V_{1},\ldots,V_{t}$ de fa\c{c}on qu'elles soient deux \`a deux disjointes (il suffit de choisir des \'el\'ements~$v$ assez petits). Pour $i\in\cn{1}{t}$, posons $V''_{i} = V_{i}\setminus V'_{i}$. C'est une partie ferm\'ee de~$Y(B)$. Posons
$$Y'(B) = Y(B)\setminus \bigcup_{1\le i\le t} V''_{i}$$ 
et consid\'erons le $G$-rev\^etement $\textrm{Ind}_{\la e \ra}^G(Y'(B))$ induit par le rev\^etement trivial au-dessus de~$Y'(B)$. 

Raccordons, \`a pr\'esent, les diff\'erents rev\^etements selon les relations entre les \'el\'ements du groupe~$G$, par la m\'ethode d\'ecrite au num\'ero~\ref{sectioncomplexe}. Nous obtenons un rev\^etement de~$Y(B)$, galoisien de groupe~$G$. On montre \`a l'aide du lemme~\ref{lemintlocalberko} et de son analogue dans le cas des rev\^etements d'Artin-Schreier-Witt, qu'il est int\`egre. Un th\'eor\`eme du type GAGA (\emph{cf.} corollaire~\ref{GAGAvoiscompact}) assure qu'il est alg\'ebrique. En passant aux corps de fonctions, nous obtenons donc finalement une extension finie et galoisienne de groupe~$G$
$$\Ms(Y(B)) = \Frac(\Os(B))(T) \to L,$$
o\`u~$\Ms$ d\'esigne le faisceau des fonctions m\'eromorphes. La construction que nous avons men\'ee \'etant purement g\'eom\'etrique, on se convainc ais\'ement que cette extension est r\'eguli\`ere, c'est-\`a-dire que le corps~$\Frac(\Os(B))$ est alg\'ebriquement ferm\'e dans le corps~$L$. 

\begin{thm}\label{galoisberko}
Il existe une extension finie du corps $\Frac(\Os(B))(T)$ qui est r\'e\-gu\-li\`ere et galoisienne de groupe de Galois~$G$.
\end{thm}

\subsection{Conclusion}\label{conclusion}

Regroupons les r\'esultats que nous avons obtenus jusqu'ici.

\begin{thm}\label{galoisinverse}
Soient~$k$ un corps muni d'une valeur absolue ultram\'etrique pour laquelle il est complet. Soient~$X$ un espace $k$-analytique et~$B$ une partie compacte de~$X$ qui poss\`ede un syst\`eme fondamental de voisinages affino\"ides g\'eom\'etriquement int\`egres. Supposons que pour tout nombre premier~$q$ diff\'erent de la caract\'eristique du corps~$k$ et tout entier $r\in\N^*$, il existe une famille infinie~$\Ks_{q^r}$ de corps deux \`a deux non isomorphes satisfaisant les propri\'et\'es suivantes :
\begin{enumerate}[\it i)]
\item tout \'el\'ement de~$\Ks_{q^r}$ est une extension finie et s\'eparable de~$k$ contenant une racine primitive $(q^r)\eme$ de l'unit\'e ;
\item pour tout \'el\'ement~$K$ de~$\Ks_{q^r}$, il existe un \'el\'ement~$x$ de~$B$ et un \'el\'ement~$\alpha$ de~$\Os(B)$ nul en tout point de~$B$ tels que les corps~$K$ et~$\Frac(\Os_{x})$ soient lin\'eairement disjoints et le polyn\^ome $S^{q^r}-\alpha$ soit irr\'eductible sur leur compositum.
\end{enumerate} 
Si la caract\'eristique du corps~$k$ est un nombre premier~$p$, supposons en outre qu'il existe une famille infinie~$\Ks_{p}$ de corps deux \`a deux non isomorphes satisfaisant les propri\'et\'es suivantes :
\begin{enumerate}[\it i)]
\item tout \'el\'ement de~$\Ks_{p}$ est une extension finie et s\'eparable de~$k$ ;
\item pour tout \'el\'ement~$K$ de~$\Ks_{p}$, il existe un \'el\'ement~$x$ de~$B$  et un \'el\'ement~$\alpha$ de~$\Os(B)$ nul en tout point de~$B$ tels que les corps~$K$ et~$\Frac(\Os_{x})$ soient lin\'eairement disjoints et le polyn\^ome $S^{p}-\alpha$ soit irr\'eductible sur leur compositum.
\end{enumerate} 

Alors, tout groupe fini est groupe de Galois d'une extension finie et r\'eguli\`ere du corps $\Frac(\Os(B))(T)$. 
\end{thm}

\begin{rem}
Ce th\'eor\`eme ne contient malheureusement aucun r\'esultat qui ne soit d\'ej\`a connu. En effet, le corps $\Frac(\Os(B))$ contient toujours un corps complet pour une valeur absolue non triviale (un corps de s\'eries de Laurent engendr\'e par l'un des \'el\'ements~$\alpha$ lorsque~$k$ est trivialement valu\'e) et, sur un tel corps, le r\'esultat est d\^u \`a D.~Harbater (\emph{cf.}~\cite{arithline}, corollary~2.4). 
\end{rem}

Nous allons, \`a pr\'esent, appliquer ce r\'esultat g\'en\'eral dans des cas particuliers.

\begin{cor}\label{corcomplet}
Soit~$k$ un corps muni d'une valeur absolue ultram\'etrique non triviale pour laquelle il est complet. Alors, tout groupe fini est groupe de Galois d'une extension finie et r\'eguli\`ere du corps $k(T)$. 
\end{cor}
\begin{proof}
Appliquons le th\'eor\`eme~\ref{galoisinverse} en choisissant pour espace~$X$ la droite analytique~$\E{1}{k}$, dont nous noterons~$U$ la variable, et pour partie~$B$ le point~$0$. Soit~$K$ une extension finie du corps~$k$. Choisissons pour point~$x$ le point~$0$ : l'anneau local~$\Os_{0}$ est l'anneau des s\'eries en une variable \`a coefficients dans~$k$ de rayon de convergence strictement positif. Par cons\'equent, les corps~$K$ et~$\Frac(\Os_{0})$ sont lin\'eairement disjoints sur~$k$. Choisissons pour fonction~$\alpha$ la fonction~$U$ : pour tout entier~$n\ge 1$, le polyn\^ome $S^n-U$ est irr\'eductible sur le corps $\Frac(\Os_{0})\otimes_{k} K$, qui est un sous-corps de~$K(\!(U)\!)$, par le th\'eor\`eme d'Eisenstein. Les hypoth\`eses du th\'eor\`eme~\ref{galoisinverse} sont donc satisfaites.

Soit~$G$ un groupe fini. Il existe une extension finie et r\'eguli\`ere du corps $\Frac(\Os(B))(T)=\Frac(\Os_{0})(T)$ qui est galoisienne de groupe~$G$. Puisque l'anneau local~$\Os_{0}$ est compos\'e de s\'eries convergentes et que le corps~$k$ n'est pas trivialement valu\'e, toute vari\'et\'e qui poss\`ede un point sur $\Frac(\Os_{0})$ en poss\`ede un sur~$k$. En utilisant le th\'eor\`eme de Bertini-Noether (\emph{cf.}~\cite{Fried-Jarden}, proposition~10.4.2), on montre alors que l'on peut sp\'ecialiser l'extension pr\'ec\'edente en une extension finie et r\'eguli\`ere du corps $k(T)$ qui est galoisienne de groupe~$G$.
\end{proof}

Rappelons qu'un corps~$k$ est dit fertile\footnote{Nous empruntons ce terme \`a L.~Moret-Bailly (\emph{cf.}~\cite{courbespointees}). Les corps fertiles sont connus sous beaucoup d'autres noms. Ils ont \'et\'e introduits par F.~Pop dans~\cite{largefields} sous l'appellation de \og large fields \fg.} si tout $k$-sch\'ema de type fini qui poss\`ede un point sur~$k(\!(U)\!)$ en poss\`ede un sur~$k$ (l'on peut d\'emontrer que cela \'equivaut \`a demander que toute $k$-courbe lisse qui poss\`ede un point sur~$k$ en poss\`ede une infinit\'e). F.~Pop a d\'emontr\'e que, si~$k$ est un corps fertile, tout groupe fini est groupe de Galois d'une extension finie et r\'eguli\`ere du corps~$k(T)$ (\emph{cf.}~\cite{largefields}, main theorem~A).

\begin{cor}\label{corfertile}
Soit~$k$ un corps. Alors, tout groupe fini est groupe de Galois d'une extension finie et r\'eguli\`ere du corps $k(\!(U)\!)(T)$. En particulier, si le corps~$k$ est fertile ou contient un corps fertile, tout groupe fini est groupe de Galois d'une extension finie et r\'eguli\`ere du corps $k(T)$. 
\end{cor}
\begin{proof}
Munissons le corps~$k$ de la valuation triviale et consid\'erons la m\^eme situation que dans la preuve pr\'ec\'edente. Nous avons alors $\Os(B)=\Os_{0}=k[\![U]\!]$ et le th\'eor\`eme~\ref{galoisinverse} fournit le r\'esultat annonc\'e. 

Lorsque le corps~$k$ est fertile, le th\'eor\`eme de Bertini-Noether permet de sp\'ecialiser l'extension pr\'ec\'edente en une extension de~$k(T)$ poss\'edant les m\^emes propri\'et\'es. La r\'egularit\'e de l'extension de~$k(T)$ permet d'obtenir, par produit tensoriel, pour tout corps~$L$ contenant~$k$, une extension de~$L(T)$ poss\'edant encore les propri\'et\'es requises.
\end{proof}

\begin{rem}
Le r\'esultat du second corollaire contient le r\'esultat du premier, puisque tout corps complet pour une valeur absolue ultram\'etrique non triviale est fertile. Cependant, la preuve de ce dernier \'enonc\'e \'etant assez difficile (on peut, par exemple, le d\'emontrer en utilisant l'approximation d'Artin), nous avons choisi de proposer une preuve directe du corollaire~\ref{corcomplet}.  
\end{rem}

\subsection{Cas de la valuation triviale}

Lorsque le corps~$k$ est trivialement valu\'e, nous pouvons obtenir des r\'esultats plus g\'en\'eraux. Nous indiquons simplement ici les modifications \`a apporter au raisonnement qui pr\'ec\`ede.

Supposons que le corps~$k$ est muni de la valeur absolue triviale. Pour tout $n\in\N$, nous pouvons alors d\'efinir une application, appel\'ee flot (\emph{cf.}~\cite{monjolimemoire}, 1.3), de $\E{n}{k}\times\R_{+}^*$ dans~$\E{n}{k}$ de la fa\c{c}on suivante. Soient~$x$ un point de~$\E{n}{k}$ --~il est associ\'e \`a une semi-norme multiplicative~$|.|_{x}$ sur $k[T_{1},\ldots,T_{n}]$ qui induit la valeur absolue triviale sur~$k$~-- et~$\eps$ un nombre r\'eel strictement positif.  L'image du couple $(x,\eps)$ est le point de~$\E{n}{k}$ associ\'ee \`a la semi-norme multiplicative~$|.|_{x}^\eps$. Par restriction \`a la source et au but, nous pouvons encore d\'efinir le flot sur tout ferm\'e de Zariski d'un espace affine analytique. Signalons qu'une fonction d\'efinie au voisinage d'un point se prolonge, et ce de fa\c{c}on unique, \`a un voisinage de sa trajectoire sous le flot (\emph{ibid.}, proposition 1.3.10). 

Nous consid\'ererons d\'esormais un espace analytique~$X$ qui est un ferm\'e de Zariski d'un espace affine analytique et une partie ouverte~$B$ de~$X$. Nous d\'efinissons comme pr\'ec\'edemment~$Y$, $\pi$ et~$\lambda$.

\bigskip

Expliquons comment adapter les constructions locales du num\'ero~\ref{constructionlocaleberko}. Comme alors, choisissons une extension finie et s\'eparable~$K$ de~$k$. Soit $P\in k[T]$ le polyn\^ome minimal unitaire d'un \'el\'ement primitif de cette extension et notons~$t$ le point de~$\P^{1,\an}_{k}$ d\'efini par l'annulation de ce polyn\^ome. 

Reprenons, \`a pr\'esent, le raisonnement du num\'ero~\ref{Kummer}. \`A cet effet, choisissons un entier~$n$ sup\'erieur ou \'egal \`a~$2$ et premier \`a~$p$ et supposons que le corps~$K$ contient une racine primitive~$n\eme$ de l'unit\'e. Par hens\'elianit\'e, elle se rel\`eve, dans l'anneau local au point~$t$, en une racine primitive $n\eme$ de l'unit\'e, que nous noterons~$\zeta$. Les propri\'et\'es du flot assurent qu'elle est d\'efinie sur l'ouvert  
$$\left\{\left.y\in\P^{1,\an}_{k}\,\right|\, |P(y)|<1\right\}.$$
Soit~$\alpha$ un \'el\'ement de~$\Os(B)$. Insistons sur le fait que nous ne supposons plus qu'il soit nul en tout point de~$B$. Nous d\'efinissons alors un faisceau~$\Fs$, comme pr\'ec\'edemment, au-dessus de l'ouvert
$$V_{t}(B) = \left\{y\in Y(B)\, \big|\, |P(y)|<1\right\}.$$
Puisque le corps~$k$ est trivialement valu\'e, le rayon de convergence de la s\'erie $(1+T)^{1/n}$ est \'egal \`a~$1$ et le rev\^etement associ\'e \`a~$\Fs$ est trivial au-dessus de l'ouvert
$$V_{t}'(B) = \left\{y\in V(B)\, \big|\, |\alpha(y)| < |P(y)|^n\right\}.$$
Supposons, en outre, que l'\'el\'ement~$\alpha$ est de valeur absolue strictement inf\'erieure \`a~$1$ en tout point de~$B$. Alors, le compl\'ementaire de la partie~$V_{t}'(B)$ dans~$V_{t}(B)$ est ferm\'e dans~$Y(B)$.

Pour assurer l'irr\'eductibilit\'e du rev\^etement associ\'e au faisceau~$\Fs$, nous rempla\c{c}ons la condition~(CGI) par la condition suivante : l'ouvert~$B$ est limite inductive d'espaces affino\"{\i}des g\'eom\'etriquement int\`egres. Le r\'esultat du lemme~\ref{lemintlocalberko} vaut alors encore.

\bigskip

Passons aux r\'esultats du num\'ero~\ref{Artin-Schreier-Witt}. Supposons donc que $p$ est un nombre premier et consid\'erons un entier~$n$ de la forme~$p^r$, avec $r\in\N^*$. Soit~$\alpha$ un \'el\'ement de~$\Os(B)$ et posons, de nouveau,
$$V_{t}(B) = \left\{y\in Y(B)\, \big|\, |P(y)|<1\right\}.$$
Les propri\'et\'es du flot permettent de pr\'eciser le domaine de d\'efinition des racines du polyn\^ome~$P_{0}(S_{0})$. Notons~$B_{0}$ le lieu d'annulation de~$\alpha$ dans~$B$. Supposons qu'il ne soit pas vide et que pour tout point~$b$ de~$B$ v\'erifiant $|\alpha(b)|<1$ et tout voisinage~$B_{+}$ de~$B_{0}$, le flot joigne le point~$b$ \`a un point de~$B_{+}$ (c'est en particulier le cas d\`es que la partie~$B$ est stable par le flot, que l'\'el\'ement~$\alpha$ est de valeur absolue strictement inf\'erieure \`a~$1$ en tout point de~$B$ et s'annule sur~$B$). Posons
$$V'_{t}(B) = \left\{y\in V(B)\, \big|\, |\alpha(y)|<1 \textrm{ et } |P(y)| < 1\right\}.$$
Les propri\'et\'es du flot assurent que le rev\^etement associ\'e \`a~$\Fs$ est encore trivial sur une partie~$V'_{t}(B)$ de~$V_{t}(B)$ dont le compl\'ementaire dans~$V_{t}(B)$ est ferm\'e. Si l'\'el\'ement~$\alpha$ est de valeur absolue strictement inf\'erieure \`a~$1$ en tout point de~$B$, nous pouvons m\^eme choisir la partie~$V'_{t}(B)$ de fa\c{c}on que son compl\'ementaire dans~$V_{t}(B)$ soit ferm\'e dans~$Y(B)$.

\bigskip

Dans la preuve du th\'eor\`eme~\ref{galoisberko}, il faut finalement remplacer le r\'esultat de type GAGA du corollaire~\ref{GAGAvoiscompact} par celui du corollaire~\ref{GAGAsurouvert}.

Nous obtenons finalement le r\'esultat suivant :

\begin{thm}\label{galoisinversetriv}
Soient~$k$ un corps. Munissons-le de la valeur absolue triviale. Soient~$X$ un espace de Zariski d'un espace affine $k$-analytique et~$B$ une partie ouverte de~$X$ stable par le flot qui soit limite inductive d'espaces affino\"{\i}des g\'eom\'etriquement int\`egres. Supposons que pour tout nombre premier~$q$ diff\'erent de la caract\'eristique du corps~$k$ et tout entier $r\in\N^*$, il existe une famille infinie~$\Ks_{q^r}$ de corps deux \`a deux non isomorphes satisfaisant les propri\'et\'es suivantes :
\begin{enumerate}[\it i)]
\item tout \'el\'ement de~$\Ks_{q^r}$ est une extension finie et s\'eparable de~$k$ contenant une racine primitive $(q^r)\eme$ de l'unit\'e ;
\item pour tout \'el\'ement~$K$ de~$\Ks_{q^r}$, il existe un \'el\'ement~$x$ de~$B$ et un \'el\'ement~$\alpha$ de~$\Os(B)$ de valeur absolue strictement inf\'erieure \`a~$1$ en tout point de~$B$ et qui s'annule sur~$B$ tels que les corps~$K$ et~$\Frac(\Os_{x})$ soient lin\'eairement disjoints et le polyn\^ome $S^{q^r}-\alpha$ soit irr\'eductible sur leur compositum.
\end{enumerate} 
Si la caract\'eristique du corps~$k$ est un nombre premier~$p$, supposons en outre qu'il existe une famille infinie~$\Ks_{p}$ de corps deux \`a deux non isomorphes satisfaisant les propri\'et\'es suivantes :
\begin{enumerate}[\it i)]
\item tout \'el\'ement de~$\Ks_{p}$ est une extension finie et s\'eparable de~$k$ ;
\item pour tout \'el\'ement~$K$ de~$\Ks_{p}$, il existe un \'el\'ement~$x$ de~$B$  et un \'el\'ement~$\alpha$ de~$\Os(B)$ de valeur absolue strictement inf\'erieure \`a~$1$ en tout point de~$B$ et qui s'annule sur~$B$ tels que les corps~$K$ et~$\Frac(\Os_{x})$ soient lin\'eairement disjoints et le polyn\^ome $S^{p}-\alpha$ soit irr\'eductible sur leur compositum.
\end{enumerate} 

Alors, tout groupe fini est groupe de Galois d'une extension finie et r\'eguli\`ere du corps $\Frac(\Os(B))(T)$. 
\end{thm}

\begin{rem}
De nouveau, le r\'esultat de ce th\'eor\`eme est connu, puisque le corps $\Frac(\Os(B))$ contient un corps de s\'eries de Laurent sur~$k$ (engendr\'e par l'un des \'el\'ements~$\alpha$). 
\end{rem}

\begin{cor}\label{corvaltriv}
Soit~$k$ un corps. Notons~$k_{+\infty,1^-}[\![U_{1},U_{2}]\!]$ le sous-anneau de $k[U_{1}][\![U_{2}]\!]$ compos\'e des s\'eries de la forme 
$$\sum_{n\ge 0} a_{n}(U_{1})U_{2}^n$$
qui v\'erifient la condition
$$\forall r>0,\forall s\in\of{[}{0,1}{[},\ \lim_{n\to +\infty}  (r^{\deg(a_{n})}\, s^n) = 0.$$
Alors, tout groupe fini est groupe de Galois d'une extension finie et r\'eguli\`ere du corps $\Frac(k_{+\infty,1^-}[\![U_{1},U_{2}]\!])(T)$. 
\end{cor}
\begin{proof}
Appliquons le th\'eor\`eme~\ref{galoisinverse} en choisissant pour espace~$X$ l'espace analytique de dimension deux~$\E{2}{k}$, dont nous noterons~$U_{1}$ et~$U_{2}$ les variables, et pour partie~$B$ le disque ouvert relatif de rayon~$1$ au-dessus de~$\E{1}{k}$ :
$$B = \left\{\left.x\in\E{2}{k}\,\right|\, |U_{2}(x)|<1\right\}.$$
Cette partie est stable par le flot et nous avons $\Os(B) = k_{+\infty,1^-}[\![U_{1},U_{2}]\!]$.

Soit~$K$ une extension finie du corps~$k$. Choisissons pour point~$x$ le point de coordonn\'ees~$(0,0)$ : l'anneau local~$\Os_{x}$ est isomorphe \`a~$k[\![U_{1},U_{2}]\!]$. Par cons\'equent, les corps~$K$ et~$\Frac(\Os_{x})$ sont lin\'eairement disjoints sur~$k$. Choisissons pour fonction~$\alpha$ la fonction~$U_{2}$ : pour tout entier~$n\ge 1$, le polyn\^ome $S^n-U_{2}$ est irr\'eductible sur le corps $\Frac(\Os_{x})\otimes_{k} K \simeq K[\![U_{1},U_{2}]\!]$. Les hypoth\`eses du th\'eor\`eme~\ref{galoisinverse} sont donc satisfaites. On en d\'eduit le r\'esultat attendu.
\end{proof}

Ce dernier \'enonc\'e peut surprendre, puisqu'il d\'ecoule directement du corollaire~\ref{corfertile}. En effet, le corps $\Frac(k_{r^-,1^-}[\![U_{1},U_{2}]\!])$ contient le corps $k(\!(U_{2})\!)$. Il pr\'esente cependant un int\'er\^et dans le cadre de l'analogie entre corps de fonctions et corps de nombres. La droite analytique sur un corps trivialement valu\'e est proche, \`a bien des \'egards, du spectre analytique --~espace analytique de dimension~$0$~-- d'un anneau d'entiers de corps de nombres (\emph{cf.} annexe~\ref{sectionberko}). Il semble donc raisonnable d'envisager que le r\'esultat du corollaire pr\'ec\'edent reste vrai en rempla\c{c}ant l'anneau $k_{r^-,1^-}[\![U_{1},U_{2}]\!]$ par l'anneau du disque ouvert de rayon~$1$ au-dessus du spectre d'un anneau d'entiers de corps de nombres. Signalons que les constructions effectu\'ees peuvent effectivement \^etre men\'ees dans ce cadre. Pour conclure, seuls manquent encore les th\'eor\`emes du type GAGA sur les espaces de Berkovich au-dessus des anneaux d'entiers de corps de nombres.

\begin{conj}
Soient~$K$ un corps de nombres, $A$ l'anneau de ses entiers et~$\Sigma_{\infty}$ l'ensemble des plongements de~$K$ dans~$\C$. Notons $A_{1^-}[\![X]\!]$ le sous-anneau de $A[\![X]\!]$ compos\'e des s\'eries~$f$ telles que, pour tout $\sigma\in\Sigma_{\infty}$, la s\'erie \`a coefficients complexes~$\sigma(f)$ a un rayon de convergence sup\'erieur ou \'egal \`a~$1$. Alors, tout groupe fini est groupe de Galois d'une extension finie et r\'eguli\`ere du corps $\Frac(A_{1^-}[\![X]\!])(T)$. 
\end{conj}

\begin{rem}
Nous ignorons si le corps $\Frac(A_{1^-}[\![X]\!])$ est fertile, m\^eme lorsque l'anneau~$A$ est l'anneau~$\Z$ des entiers. 
\end{rem}

\section{Probl\`eme inverse de Galois sur un disque relatif}\label{sectiondemo}

Soit~$A$ un anneau d'entiers de corps de nombres. Nous allons maintenant appliquer la strat\'egie de raccord d\'ecrite \`a la section~\ref{sectioncomplexe} dans le cadre des espaces de Berkovich sur~$A$ (\emph{cf.} annexe~\ref{sectionberko}). Plus pr\'ecis\'ement, un groupe fini~$G$ \'etant donn\'e, nous allons construire un rev\^etement galoisien de groupe~$G$ du disque
$$\D=\left\{\left. x\in\E{1}{A}\, \right|\, |T(x)|<1\right\}.$$
Cela indique que la troisi\`eme \'etape de notre d\'emonstration diff\'erera fondamentalement de la troisi\`eme \'etape de la d\'emonstration du cas complexe. En effet, le disque~$\D$ est un espace affine et non plus projectif comme l'\'etait~$\P^1(\C)$. En particulier, les th\'eor\`emes GAGA n'y sont pas valables. Nous utiliserons, pour les remplacer, le caract\`ere Stein du disque~$\D$.


Nous noterons~$X=\E{1}{A}$ et, pour tout id\'eal maximal~$\m$ de~$A$,
$$\D_{\m} = \D \cap \pi^{-1}(\of{[}{a_{0},\tilde{a}_{\m}}{]})
\textrm{ et } \D'_{\m} = \D \cap \pi^{-1}(\of{]}{a_{0},\tilde{a}_{\m}}{]}).$$
Ces deux parties sont connexes.

\subsection{Construction locale de rev\^etements cycliques}\label{constructionlocale}

Dans le cas complexe, la construction locale \'etait particuli\`erement simple car nous disposions de racines primitives de l'unit\'e de tout ordre. Elle ne sera gu\`ere plus difficile ici puisque, comme nous allons l'expliquer, un entier~$n\ge 1$ \'etant donn\'e, il existe toujours une branche de~$\Mc(A)$, et m\^eme une infinit\'e, dont l'anneau des fonctions contient une racine primitive~$n^\textrm{\`eme}$ de l'unit\'e.

Soient~$n\ge 1$ un entier, $p$ un nombre premier congru \`a~$1$ modulo~$n$ et~$\m$ un id\'eal maximal de~$A$ contenant~$p$. Notons~$\hat{A}_{\m}$ le compl\'et\'e de l'anneau~$A$ pour la topologie $\m$-adique. Soit~$\pi_{\m}$ une uniformisante de l'anneau~$\hat{A}_{\m}$. Posons
$$Q(S) = S^n-\pi_{\m}^n-T \in \Os(\D'_{\m})[S].$$
D\'efinissons un pr\'e\-fais\-ceau~$\Fs$ sur~$D'_{\m}$ en posant, pour toute partie ouverte~$W$ de~$D'_{\m}$,
$$\Fs(W) = \Os(W)[S]/(Q(S))$$
et en utilisant les morphismes de restriction induits par ceux du faisceau~$\Os$. Le caract\`ere unitaire du polyn\^ome~$Q$ assure que~$\Fs$ est un faisceau de $\Os_{D'_{\m}}$-alg\`ebres coh\'erent. Nous consid\'erons ce faisceau comme l'image directe du faisceau d'un rev\^etement fini de~$\D'_{\m}$.

Le r\'esultat classique qui suit explique le choix des entiers~$n$ et~$p$.

\begin{lem}
L'anneau~$\Z_{p}$ contient une racine primitive~$n\eme$ de l'unit\'e et, pour tout $i\in\N$, nous avons $C_{1/n}^i \in \Z_{p}$.
\end{lem}

Soit~$\zeta\in\hat{A}_{\m}$ une racine primitive $n^\textrm{\`eme}$ de l'unit\'e. Posons
$$U =  \left\{x\in \D'_{p}\, \big|\, |T(x)|< |\pi_{\m}(x)|^n  \right\}.$$
Le r\'esultat suivant affirme que le rev\^etement de~$\D'_{\m}$ associ\'e au faiseau~$\Fs$ est trivial au-dessus de l'ouvert~$U$.

\begin{prop}\label{nonramifie}
Il existe un isomorphisme de $\Os_{U}$-alg\`ebres
$$\varphi : \Fs \to \Os^n$$
tel que, pour tout ouvert~$V$ de~$U$ et tout \'el\'ement~$s$ de~$\Fs(V)$, nous ayons
$$\varphi(\zeta s) = \tau(\varphi(s)),$$
o\`u~$\tau$ d\'esigne l'automorphisme du faisceau~$\Os^n$ qui consiste \`a faire agir la permutation cyclique $(1\ 2\ \cdots\ n)$ sur les coordonn\'ees.
\end{prop}
\begin{proof}
En utilisant le lemme pr\'ec\'edent, on montre que la fonction $\pi_{\m}^{-n}\, T$ poss\`ede une racine~$n\eme$ dans~$\Os(U)$. Nous la noterons~$\omega$. On en d\'eduit l'\'egalit\'e
$$Q(S) = S^n-\pi_{\m}^n-T = \prod_{j=0}^n (S - \pi_{\m}\, \zeta^{j}\, \omega) \textrm{ dans } \Os(U)[S].$$
Par cons\'equent, le morphisme
$$\begin{array}{ccc}
\Fs & \to & \Os^n\\
R(S) & \mapsto & \left(R(\pi_{\m}\, \omega), R(\pi_{\m}\, \zeta^{-1}\, \omega),\ldots, R(\pi_{\m}\, \zeta^{-(n-1)}\, \omega)\right) 
\end{array}$$
est un isomorphisme. On v\'erifie imm\'ediatement qu'il satisfait la condition requise.
\end{proof}

D\'emontrons, \`a pr\'esent, que le rev\^etement est irr\'eductible.

\begin{lem}\label{lemintlocal}
Le polyn\^ome~$Q(S)=S^n-\pi_{\m}^n-T$ est irr\'eductible sur le corps Frac$(\Os(\D'_{\m}))$. En particulier, l'anneau $\Fs(\D'_{\m})$ est int\`egre.
\end{lem}
\begin{proof} 
Notons~$z_{\m}$ le point~$0$ de la fibre $\pi^{-1}(\tilde{a}_{\m})$. D'apr\`es la discussion men\'ee \`a la fin de la section~\ref{sectionberko}, l'anneau local en ce point est isomorphe \`a l'anneau~$\hat{A}_{\m}[\![T]\!]$. Commen\c{c}ons par montrer que le polyn\^ome~$Q(S)$ est ir\-r\'e\-duc\-ti\-ble sur le corps $\Frac(\Os_{z_{\m}})$. Pour des raisons de valuation $T$-adique, l'\'el\'ement \mbox{$\pi_{\m}^n+T$} de~$\hat{A}_{\m}[\![T]\!]$ n'est racine $d^\textrm{\`eme}$ dans~$\hat{A}_{\m}[\![T]\!]$ pour aucun diviseur~\mbox{$d\ge 2$} de~$n$. D'apr\`es la th\'eorie de Kummer, cela impose au polyn\^ome~$Q(S)$ d'\^etre irr\'eductible sur~$\Frac(\Os_{z_{\m}})$. Les m\^emes arguments que dans la preuve du lemme~\ref{lemintlocalberko} permettent alors de conclure.
\end{proof}

Nous pouvons m\^eme \^etre plus pr\'ecis et d\'emontrer un principe du prolongement analytique.

\begin{lem}\label{integreunebranche}
Soient~$x$ un point de~$U$ et~$i$ un \'el\'ement de~$\cn{1}{n}$. Le morphisme
$$\rho_{i,x} : \Fs(\D'_{\m}) \to \Fs_{x} \xrightarrow[\sim]{\varphi_{x}} \Os_{x}^n \xrightarrow[]{p_{i}} \Os_{x},$$
o\`u~$p_{i}$ est la projection sur le $i^\textrm{\`eme}$ facteur, est injectif.
\end{lem}
\begin{proof}
Soit~$s$ un \'el\'ement de l'anneau $\Fs(\D'_{\m})=\Os(\D'_{\m})[S]/(Q(S))$ dont l'image par le morphisme~$\rho_{i,x}$ est nulle. Choisissons un \'el\'ement~$R(S)$ de $\Os(\D'_{\m})[S]$ qui repr\'esente la section~$s$. Reprenons les notations de la preuve de la proposition~\ref{nonramifie}. Par hypoth\`ese, nous avons
$$R(\pi_{\m}\, \zeta^{-i}\, \omega)=0 \textrm{ dans } \Os_{x}.$$
Pour montrer que l'\'el\'ement~$s$ est nul, il suffit de montrer que le polyn\^ome~$Q(S)$ est le polyn\^ome minimal de l'\'el\'ement $\pi_{\m}\, \zeta^{-i}\, \omega$ sur le corps~$\Frac(\Os(\D'_{\m}))$. C'est bien le cas, puisque le lemme pr\'ec\'edent assure que le polyn\^ome~$Q$ est irr\'eductible sur le corps~$\Frac(\Os(\D'_{\m}))$.
\end{proof}


Terminons par un r\'esultat topologique. 

\begin{lem}\label{Festferme}
La partie 
$$F = \D'_{\m} \setminus U = \left\{x\in \D'_{\m}\, \big|\, |T(x)|\ge |\pi_{\m}(x)|^n \right\}$$
est ferm\'ee dans le disque~$\D$.
\end{lem}
\begin{proof}
Il suffit de montrer que~$F$ est ferm\'ee dans~$\D_{\m}$ puisque cette derni\`ere partie est elle-m\^eme ferm\'ee dans~$\D$. En d'autres termes, nous souhaitons montrer que la partie 
$$V = U \cup (\D\cap \pi^{-1}(a_{0}))$$
est ouverte dans~$\D_{\m}$. Puisque~$U$ est une partie ouverte de~$\D_{\m}$, il suffit de montrer que~$V$ est voisinage dans~$\D_{\m}$ de chacun des points de~$\D\cap \pi^{-1}(a_{0})$.

Soit~$x$ un point de~$\D\cap \pi^{-1}(a_{0})$. Posons $r=|T(x)|$. C'est un \'el\'ement de l'intervalle~$\of{]}{0,1}{[}$. Soient~$s$ un \'el\'ement de~$\of{]}{r,1}{[}$ et~$\eps$ un \'el\'ement de~$\of{]}{0,1}{[}$ tels que l'on ait $|\pi_{\m}|_{\m}^{n\eps} > s$. La partie 
$$\left\{y\in \pi^{-1}({[}{a_{0},a_{\m}^\eps}{[}) \, \big|\, |T(y)|<s \right\}$$
est un voisinage ouvert du point~$x$ dans~$\D_{\m}$ qui est contenu dans~$V$. 
\end{proof}

\subsection{Raccord et retour \`a l'alg\`ebre}\label{numerorecollement}

Soit~$G$ un groupe fini. Notons~$n\in\N^*$ son ordre. Soient $g_{1},\ldots,g_{t}$, avec~$t\in\N^*$, des g\'en\'erateurs du groupe~$G$. Pour tout \'el\'ement~$i$ de~$\cn{1}{t}$, notons~$n_{i}$ l'ordre de l'\'el\'ement~$g_{i}$, choisissons un nombre premier~$p_{i}$ congru \`a~$1$ modulo~$n_{i}$ et un id\'eal maximal~$\m_{i}$ de~$A$ contenant~$p_{i}$. Nous pouvons supposer que les~$\m_{i}$ sont distincts.

Soit~$i$ un \'el\'ement de~$\cn{1}{t}$. Construisons par la m\'ethode du num\'ero~\ref{constructionlocale} un rev\^etement galoisien de groupe~$\Z/n_{i}\Z$. Il est d\'efini au-dessus de~$\D'_{\m_{i}}$ et trivial au-dessus de
$$U_{i} = \left\{x\in \D'_{\m_{i}}\, \big|\, |T(x)| < |\pi_{\m_{i}}(x)|^{n_{i}} \right\}.$$
Notons $\textrm{Ind}_{\la g_{i} \ra}^G(\D'_{\m_{i}})$ le $G$-rev\^etement induit (\emph{cf.} section~\ref{sectioncomplexe}). 

D'apr\`es le lemme \ref{Festferme}, pour tout \'el\'ement~$i$ de~$\cn{1}{t}$, la partie $F_{i}=\D'_{\m_{i}}\setminus U_{i}$ est ferm\'ee dans~$\D$. D\'efinissons une partie ouverte de~$\D$ par 
$$U_{0} = \D\setminus \bigcup_{1\le i\le t} F_{i}.$$
On se convainc ais\'ement qu'elle est connexe. Consid\'erons le $G$-rev\^etement $\textrm{Ind}_{\la e \ra}^G(U_{0})$ induit par le rev\^etement trivial au-dessus de~$U_{0}$. Recollons ces diff\'erents rev\^etements par la m\'ethode d\'ecrite au num\'ero~\ref{sectioncomplexe}. Nous obtenons un rev\^etement de~$\D$, galoisien de groupe~$G$, associ\'e \`a un faisceau~$\Gs$. On montre \`a l'aide du lemme~\ref{integreunebranche} qu'il est int\`egre, c'est-\`a-dire que l'anneau~$\Gs(\D)$ est int\`egre. 

Nous disposons, \`a pr\'esent, d'un rev\^etement du disque~$\D$ poss\'edant le groupe de Galois d\'esir\'e~$G$. Il nous reste \`a montrer que l'extension induite entre les corps de fonctions est galoisienne de m\^eme groupe. Nous utiliserons, pour ce faire, le caract\`ere Stein du disque~$\D$ (\emph{cf.} th\'eor\`eme~\ref{lemniscateStein}).

\begin{prop}\label{GinjAutGsD}
Le groupe des automorphismes de $\Os(\D)$-alg\`ebres du faisceau~$\Gs(\D)$ est isomorphe \`a~$G$.
\end{prop}
\begin{proof}
Soient~$\As$ et~$\Bs$ deux faisceaux de $\Os_{\D}$-alg\`ebres coh\'erents. Consid\'erons l'application
$$\textrm{Mor}_{\Os}(\As,\Bs) \to \textrm{Mor}_{\Os(\D)}(\As(\D),\Bs(\D)).$$
Elle est bijective car les faisceaux~$\As$ et~$\Bs$ satisfont le th\'eor\`eme~A sur~$\D$.

Par construction, le groupe des automorphismes de $\Os_{\D}$-alg\`ebres du faisceau~$\Gs$ est isomorphe \`a~$G$. On en d\'eduit le r\'esultat attendu.
\end{proof}

Il reste \`a montrer que l'extension $\Os(\D) \to \Gs(\D)$ est enti\`ere. Puisque les th\'eor\`emes du type GAGA ne sont pas valables dans ce cadre, nous utiliserons un raisonnement direct.

\begin{lem}\label{racinedegreinfn}
Tout \'el\'ement de~$\Gs(\D)$ annule un polyn\^ome unitaire \`a coefficients dans~$\Os(\D)$ de degr\'e inf\'erieur \`a~$n$.
\end{lem}
\begin{proof}
Soit~$s$ un \'el\'ement de~$\Gs(\D)$. Nous supposerons, tout d'abord, qu'il existe un point~$x_{0}$ de~$U_{0}$ tel que toutes les coordonn\'ees de son image~$s_{x_{0}}$ dans $\Gs_{x_{0}}=\Os_{X,x_{0}}^n$ soient distinctes. Puisque l'ouvert~$U_{0}$ est connexe, le principe du prolongement analytique (\emph{cf.} th\'eor\`eme \ref{prolan}) assure qu'en tout point~$x$ de~$U_{0}$, toutes les coordonn\'ees du germe~$s_{x}$ sont distinctes. Notons $a_{1},\ldots,a_{n}$ les coordonn\'ees de l'image de~$s$ dans $\Gs(U_{0})=\Os(U_{0})^n$. Posons
$$M(Z) = \prod_{l=1}^n (Z-a_{l}) \in \Os(U_{0})[Z].$$
En tout point~$x$ de~$U_{0}$, l'image du polyn\^ome~$M$ est l'unique polyn\^ome unitaire de degr\'e inf\'erieur \`a~$n$ \`a coefficients dans~$\Os_{x}$ qui annule le germe~$s_{x}$. 

Pour tout \'el\'ement~$j$ de~$\cn{0}{t}$, posons $V_{j} = U_{0} \cup \bigcup_{1\le i\le j} \D'_{\m_{i}}$. Montrons, par r\'ecurrence, que pour tout \'el\'ement~$j$ de~$\cn{0}{t}$, il existe un polyn\^ome unitaire~$N_{j}$ de degr\'e~$n$ \`a coefficients dans~$\Os(V_{j})$ qui annule l'\'el\'ement~$s_{|V_{j}}$ de~$\Gs(V_{j})$. Nous avons d\'ej\`a trait\'e le cas~$j=0$. Soit maintenant un \'el\'ement~$j$ de~$\cn{0}{t-1}$ pour lequel l'hypoth\`ese de r\'ecurrence est v\'erifi\'ee. L'\'el\'ement~$s_{|\D'_{\m_{j+1}}}$ de l'anneau $\Gs(\D'_{\m_{j+1}})=\Os(\D'_{\m_{j+1}})[S]/(S^{n_{j+1}}-p_{j+1}^{n_{j+1}}-T)$ est annul\'e par un polyn\^ome unitaire~$M_{j+1}$ de degr\'e inf\'erieur \`a~$n$ \`a coefficients dans le corps~$\Os(\D'_{\m_{j+1}})$. Soit~$x$ un \'el\'ement de $U_{j+1} = \D'_{\m_{j+1}}\cap U_{0}$. Nous avons d\'emontr\'e qu'il existe un unique polyn\^ome unitaire de degr\'e inf\'erieur \`a~$n$ \`a coefficients dans~$\Os_{x}$ qui annule le germe~$s_{x}$. On en d\'eduit que les images des polyn\^omes~$N_{j}$ et~$M_{j+1}$ dans~$\Os_{x}[Z]$ co\"incident. L'ouvert~$U_{j+1}$ \'etant connexe, d'apr\`es le th\'eor\`eme~\ref{prolan}, les images de ces polyn\^omes dans $\Os(U_{j+1})[Z]$ co\"incident. On en d\'eduit que le polyn\^ome~$N_{j}$ se prolonge en un polyn\^ome unitaire~$N_{j+1}$ de degr\'e inf\'erieur \`a~$n$ \`a coefficients dans~$\Os(V_{j+1})$ qui annule l'\'el\'ement~$s_{|V_{j+1}}$ de~$\Gs(V_{j+1})$.

On d\'eduit finalement le r\'esultat attendu du cas~$j=t$.

\bigskip

Soit~$x_{0}$ un point de l'ouvert~$U_{0}$. La fibre du faisceau~$\Gs$ au point~$x_{0}$ est isomorphe \`a l'alg\`ebre $\Os_{x_{0}}^n$. D'apr\`es le th\'eor\`eme \ref{lemniscateStein}, le faisceau~$\Gs$ v\'erifie le th\'eor\`eme~A sur le disque~$\D$. On en d\'eduit qu'il existe un \'el\'ement~$s_{0}$ de~$\Gs(\D)$ dont toutes les coordonn\'ees de l'image dans la fibre $\Gs_{x_{0}}=\Os_{x_{0}}^n$ sont distinctes. 

Soit~$s$ un \'el\'ement de~$\Gs(\D)$. Il existe un \'el\'ement~$\lambda$ de~$\Os(\D)$ tel que toutes les coordonn\'ees du germe de la section $s_{1} = s+\lambda s_{0}$ au point~$x_{0}$ soient distinctes. Le raisonnement qui pr\'ec\`ede montre qu'il existe deux polyn\^omes unitaires~$P_{0}$ et~$P_{1}$ de degr\'e inf\'erieur \`a~$n$ \`a coefficients dans~$\Os(\D)$ qui annulent respectivement les sections~$s_{0}$ et~$s_{1}$. On en d\'eduit qu'il existe un polyn\^ome unitaire~$P$ de degr\'e inf\'erieur \`a~$n$ \`a coefficients dans~$\Os(\D)$ qui annule la section~$s$.
\end{proof}

\begin{lem}\label{AalgfermeGsD}
L'anneau~$A$ est alg\'ebriquement ferm\'e dans l'anneau~$\Gs(\D)$.
\end{lem}
\begin{proof}
Soit~$P$ un polyn\^ome unitaire \`a coefficients dans~$A$ sans racines dans~$A$. Supposons, par l'absurde, qu'il existe une section~$s$ de~$\Gs(\D)$ qui est racine du polyn\^ome~$P$. Notons~$z_{0}$ le point~$0$ de la fibre $\pi^{-1}(a_{0})$ de l'espace~$X$. C'est un point de l'ouvert~$U_{0}$. Notons~$a$ la premi\`ere coordonn\'ee de l'image du germe~$s_{z_{0}}$ par l'isomorphisme $\Gs_{z_{0}} \xrightarrow[]{\sim} \Os_{z_{0}}^t$. C'est un \'el\'ement de~$\Os_{z_{0}}$ qui v\'erifie l'\'egalit\'e $P(a)=0$. D'apr\`es la discussion men\'ee \`a la fin de la section~\ref{sectionberko}, l'anneau local~$\Os_{z_{0}}$ se plonge dans l'anneau~$K[\![T]\!]$. On en d\'eduit que le polyn\^ome~$P$ poss\`ede une racine dans l'anneau~$K[\![T]\!]$ et donc dans le corps~$K$. Puisque l'anneau~$A$ est alg\'ebriquement ferm\'e dans le corps~$K$, cette racine doit appartenir \`a~$A$. Nous avons abouti \`a une contradiction. On en d\'eduit le r\'esultat annonc\'e. 
\end{proof}

Introduisons une d\'efinition correspondant \`a cette propri\'et\'e.

\begin{defi}
Une extension~$L$ du corps~$\Frac(\Os(\D))$ est dite r\'eguli\`ere si le corps~$K$ est alg\'ebriquement ferm\'e dans~$L$.
\end{defi}

Regroupons, \`a pr\'esent, les r\'esultats obtenus.

\begin{prop}
L'extension de corps
$$\Frac(\Os(\D)) \to \Frac(\Gs(\D))$$
est finie de degr\'e~$n$, r\'eguli\`ere et galoisienne de groupe~$G$. 
\end{prop}
\begin{proof}
L'extension $\Frac(\Os(\D)) \to \Frac(\Gs(\D))$ est finie et de degr\'e in\-f\'e\-rieur \`a~$n$ d'apr\`es le lemme~\ref{racinedegreinfn}. Elle est r\'eguli\`ere d'apr\`es le lemme~\ref{AalgfermeGsD}. On d\'eduit de la proposition~\ref{GinjAutGsD} qu'il existe un morphisme injectif du groupe~$G$ dans le groupe des $\Frac(\Os(\D))$-automorphismes du corps Frac($\Gs(\D)$). Or le groupe~$G$ a pour cardinal~$n$. On en d\'eduit que l'extension $\Frac(\Os(\D)) \to \Frac(\Gs(\D))$ est exactement de degr\'e~$n$, qu'elle est galoisienne et que son groupe de Galois est isomorphe au groupe~$G$.
\end{proof}


\begin{rem}
Puisque le disque~$\D$ est connexe, les th\'eor\`emes~\ref{prolan} et~\ref{lemniscateStein} assurent que l'anneau des sections m\'eromorphes globales~$\Ms(\D)$ est un corps isomorphe \`a $\Frac(\Os(\D))$. L'extension $\Frac(\Os(\D)) \to \Frac(\Gs(\D))$ est donc bien l'extension obtenue \`a partir du rev\^etement du disque~$\D$ associ\'e au faisceau~$\Gs$ en passant aux corps de fonctions.
\end{rem}

\subsection{Conclusion et g\'en\'eralisations}

Regroupons, \`a pr\'esent, les r\'esultats que nous avons obtenus. Puisque nous sommes partis d'un groupe fini~$G$ arbitraire, nous avons finalement d\'emontr\'e que tout groupe fini est groupe de Galois d'une extension finie et r\'eguli\`ere du corps~$\Frac(\Os(\D))$. D'apr\`es la description de l'anneau~$\Os(\D)$ donn\'ee \`a la fin de la section~\ref{sectionberko}, ce dernier est isomorphe au corps~$\Frac(\Aun)$, o\`u $\Aun$ d\'esigne l'anneau des s\'eries en une variable \`a coefficients dans~$A$ de rayon de convergence complexe sup\'erieur ou \'egal \`a~$1$ en toute place infinie. Lorsque $A=\Z$, nous retrouvons bien ainsi le r\'esultat de D.~Harbater (\emph{cf.}~\cite{galoiscovers}, corollary~3.8) \'enonc\'e en introduction.


\begin{thm}\label{thZun}
Soit~$A$ un anneau d'entiers de corps de nombres. Tout groupe fini est groupe de Galois d'une extension finie et r\'eguli\`ere du corps $\Frac(\Aun)$. 
\end{thm}

\begin{rem}
Pour tout $r>1$, l'anneau $A_{r^-}[\![T]\!]$ des s\'eries en une variable \`a coefficients dans~$A$ de rayon de convergence complexe sup\'erieur ou \'egal \`a~$r$ en toute place infinie (une seule suffirait) est r\'eduit \`a l'anneau de polyn\^omes~$A[T]$. Si nous disposions du th\'eor\`eme pr\'ec\'edent pour un certain nombre r\'eel $r>1$, nous aurions donc r\'esolu le probl\`eme inverse de Galois g\'eom\'etrique sur~$K$.
\end{rem}

Pour finir, nous regroupons plusieurs r\'esultats proches de celui du th\'eor\`eme~\ref{thZun}. Les d\'emonstrations en sont fort similaires et nous n'indiquerons que les modifications \`a effectuer. 

Ainsi que nous l'avons d\'ej\`a signal\'e, le spectres analytique d'un anneau d'entiers de corps de nombres pr\'esente de nombreuses similitudes avec la droite analytique sur un corps trivialement valu\'e. C'est donc, tout d'abord, dans ce cadre que nous allons nous placer. Soit~$k$ un corps. Munissons-le de la valeur absolue triviale afin d'en faire un corps ultram\'etrique complet. Consid\'erons, maintenant, l'espace~$\E{2}{k}$, analogue de~$\E{1}{A}$. Nous noterons~$U$ et~$T$ les coordonn\'ees sur cet espace. 

Lorsque la caract\'eristique du corps~$k$ est nulle, pour tout entier $n\ge 1$, il existe une infinit\'e de branches de la droite $\E{1}{k}$ dont l'anneau des fonctions contient une racine primitive $n\eme$ de l'unit\'e. Posons
$$\D_{k} = \left\{x\in\E{2}{k}\, \big|\, |T(x)|<1\right\}.$$
En appliquant le raisonnement suivi dans cette section, nous d\'emontrons que tout groupe fini est groupe de Galois d'une extension finie et r\'eguli\`ere du corps~$\Frac(\Os(\D_{k}))$. Par r\'eguli\`ere, nous entendons ici que le corps~$k(U)$ est alg\'ebriquement ferm\'e dans l'extension en question.

Supposons, \`a pr\'esent, que le corps~$k$ est de caract\'eristique~$p$, o\`u~$p$ est un nombre premier. Dans ce cas, la construction des rev\^etements cycliques locaux est plus complexe. Cependant, il est possible de la mener \`a bien en faisant appel aux extensions d'Artin-Schreier-Witt, comme nous l'avons d\'ej\`a fait au num\'ero~\ref{constructionlocaleberko}. Il est plus difficile de montrer qu'un tel rev\^etement est trivial sur une partie dont le compl\'ementaire est ferm\'e dans~$\D_{k}$, mais les propri\'et\'es du flot nous permettent d'y parvenir. 

En utilisant une description explicite de l'anneau~$\Os(\D_{k})$, nous obtenons finalement le r\'esultat suivant :


\begin{thm}\label{GaloisDk}
Soit~$k$ un corps. Notons~$k_{+\infty,1^-}[\![U,T]\!]$ le sous-anneau de $k[U][\![T]\!]$ compos\'e des s\'eries de la forme 
$$\sum_{n\ge 0} a_{n}(U)T^n$$
qui v\'erifient la condition
$$\forall r>0,\forall s\in\of{[}{0,1}{[},\ \lim_{n\to +\infty}  (r^{\deg(a_{n})}\, s^n) = 0.$$
Tout groupe fini est groupe de Galois d'une extension finie et r\'eguli\`ere du corps $\Frac(k_{+\infty,1^-}[\![U,T]\!])$. 
\end{thm}


\begin{rem}
Le th\'eor\`eme pr\'ec\'edent nous permet, en particulier, de r\'ealiser, pour tout corps~$k$, tout groupe fini comme groupe de Galois sur le corps des fractions de~$k[U][\![T]\!]$. Nous \'etendons ainsi des r\'esultats de D.~Harbater (\emph{cf.}~\cite{mockcovers}, corollary~1.4 et corollary~1.5).
\end{rem}

\appendix

\section{Th\'eor\`emes GAGA relatifs sur un affino\"{\i}de}\label{annexeGAGA}

Soit~$k$ un corps muni d'une valeur absolue ultram\'etrique pour laquelle il est complet, $\As$ une alg\`ebre $k$-affino\"{\i}de et~$X$ un sch\'ema localement de type fini sur~$\As$. Dans~\cite{bleu}, 2.6, V.~Berkovich a d\'efini, de mani\`ere fonctorielle, l'analytifi\'e~$X^\an$ du sch\'ema~$X$. Il vient avec un morphisme d'espaces localement annel\'es $X^\an \to X$, qui est plat et surjectif (cette derni\`ere propri\'et\'e tombe \'evidemment en d\'efaut dans le cas complexe). \`A tout faisceau de $\Os_{X}$-modules~$\Fs$, nous pouvons associer, par r\'etrotirette, un faisceau de $\Os_{X^\an}$-modules, que nous noterons~$\Fs^\an$. Remarquons que l'analytifi\'e d'un faisceau coh\'erent est encore coh\'erent.

Dans la lign\'ee des th\'eor\`emes GAGA de J.-P.~Serre (\emph{cf.}~\cite{GAGA} et~\cite{SGA1}, expos\'e XII), nous allons nous int\'eresser aux propri\'et\'es du foncteur d'analytification lorsque l'espace~$X$ est propre. Pr\'ecis\'ement, nous allons d\'emontrer le th\'eor\`eme suivant :

\begin{thm}\label{GAGAsuraffinoide}
Soit~$k$ un corps muni d'une valeur absolue ultram\'etrique pour laquelle il est complet, $\As$ une alg\`ebre $k$-affino\"{\i}de et~$X$ un $\As$-sch\'ema propre. Alors
\begin{enumerate}[\it i)]
\item pour tout faisceau de $\Os_{X}$-modules coh\'erent~$\Fs$ et tout entier $q\in\N$, le morphisme
$$H^q(X,\Fs) \to H^q(X^\an,\Fs^\an)$$
est un isomorphisme ;
\item le foncteur d'analytification
$$\Fs \to \Fs^\an$$
induit une \'equivalence entre la cat\'egorie des $\Os_{X}$-modules coh\'erents et celle des $\Os_{X^\an}$-modules coh\'erents.
\end{enumerate}
\end{thm}

La preuve originale de J.-P.~Serre, qui concerne l'analytification complexe, peut \^etre adapt\'ee \`a notre contexte sans difficult\'es majeures. Signalons que le th\'eor\`eme pr\'ec\'edent a d'ailleurs d\'ej\`a \'et\'e obtenu par U.~K\"opf dans le cadre de la g\'eom\'etrie rigide (\emph{cf.}~\cite{Koepf}) et par A.~Ducros en g\'en\'eral, dans un texte in\'edit. Nous en r\'edigeons cependant une d\'emonstration pour la commodit\'e du lecteur, sans pr\'etendre aucunement \`a l'originalit\'e.

Comme dans le cas complexe, on se ram\`ene \`a d\'emontrer le th\'eor\`eme pour un espace~$X$ de la forme~$\P^r_{\As}$ et on utilise les r\'esultats classiques concernant les faisceaux coh\'erents sur un tel espace. Deux propri\'et\'es joueront un r\^ole essentiel dans la preuve : la platitude du morphisme $X^\an \to X$ et la finitude cohomologique des morphismes propres. Ce dernier point prend la forme du th\'eor\`eme suivant :

\begin{thm}\label{kiehlfini}
Soit~$k$ un corps muni d'une valeur absolue ultram\'etrique pour laquelle il est complet, $\As$ une alg\`ebre $k$-affino\"{\i}de et~$X$ un $\As$-espace analytique propre. Pour tout faisceau de $\Os_{X}$-modules coh\'erent et tout entier naturel~$q$, le $\As$-module $H^q(X,\Fs)$ est un $\As$-module de Banach de type fini.
\end{thm}

Ce r\'esultat a \'et\'e d\'emontr\'e par R.~Kiehl dans le cas d'un corps de valuation non triviale et d'objets strictement affino\"{\i}des (\emph{cf.}~\cite{Kiehlpropre}, Theorem~3.3). Il a \'et\'e \'etendu au cas g\'en\'eral par V.~Berkovich (\emph{cf.}~\cite{rouge}, proposition~3.3.5).

Indiquons pour finir le seul v\'eritable ajout que nous avons d\^u faire \`a la preuve de \mbox{J.-P.~Serre} (et qui figure chez A.~Ducros) : il s'agit du lemme~\ref{chgtbasekL}, un r\'esultat de changement de base, utilis\'e pour pallier le fait qu'un point de l'espace $k$-analytique~$\P^{r,\an}_{k}$ n'est pas toujours situ\'e sur un hyperplan.

\subsection{D\'emonstration}

Commen\c{c}ons par quelques r\'eductions classiques. En utilisant le lemme de Chow (\emph{cf.} \cite{EGAII}, th\'eor\`eme~5.6.1), on montre qu'il suffit de prouver le th\'eor\`eme dans le cas o\`u~$X$ est un sch\'ema projectif sur~$\As$. Dans le cas complexe, les d\'etails de l'argument figurent dans l'expos\'e XII de \cite{SGA1} ; un raisonnement en tout point semblable vaut ici. 

Si~$X$ est un sch\'ema projectif sur~$\As$, il existe une immersion ferm\'ee $\varphi : X \to \P^r_{\As}$, pour un certain entier naturel~$r$. Pour tout faisceau de $\Os_{X}$-modules coh\'erent~$\Fs$, le faisceau de $\Os_{\P^r_{\As}}$-modules $\varphi_{*}(\Fs)$, qui n'est autre que le prolongement de~$\Fs$ par z\'ero, est encore coh\'erent. On v\'erifie que cette op\'eration de prolongement commute \`a l'analytification et pr\'eserve la cohomologie. En outre, elle poss\`ede un inverse \`a gauche : la restriction \`a~$Y$. En utilisant ces propri\'et\'es, on montre qu'il suffit de prouver le th\'eor\`eme dans le cas o\`u~$X$ est un espace projectif sur~$\As$. C'est ce que nous supposerons d\'esormais.

\subsubsection{Assertion~{\it i)} lorsque $\Fs=\Os(n)$}\label{iOn}

Nous allons d\'emontrer, par r\'ecurrence sur~$r$, que, pour tout entier naturel~$r$ et tout entier relatif~$n$, l'assertion~{\it i)} du th\'eor\`eme est vraie lorsque $X = \P^r_{\As}$ et $\Fs = \Os_{X}(n)$.

Pour $r=0$, le r\'esultat d\'ecoule du th\'eor\`eme d'acyclicit\'e de Tate.

Soit~$r\in\N$ tel que le r\'esultat soit vrai pour $\P^r_{\As}$. Posons $X=\P^{r+1}_{\As}$. Soit~$t$ une section non nulle du fibr\'e $\Os_{X}(1)$ et~$Y$ l'hyperplan de~$X$ (isomorphe \`a~$\P^r_{\As}$) qu'elle d\'efinit. Nous noterons~$\Os_{Y}$ \`a la fois le faisceau structural sur~$Y$ et son prolongement par z\'ero \`a~$X$. D'apr\`es l'hypoth\`ese de r\'ecurrence, pour tout entier $q\in\N$, le morphisme
$$H^q(X,\Os_{Y}(n)) = H^q(Y,\Os_{Y}(n)) \to H^q(Y^\an,\Os_{Y^\an}(n)) = H^q(X^\an,\Os_{Y^\an}(n))$$
est un isomorphisme.

Pour tout $n\in\Z$, la multiplication par~$t$ d\'efinit une suite exacte courte
$$0 \to \Os_{X}(n-1) \to \Os_{X}(n) \to \Os_{Y}(n) \to 0.$$
En \'ecrivant la suite exacte longue associ\'ee et en utilisant le lemme des cinq, on montre que l'on a un isomorphisme $H^q(X,\Os_{X}(n)) \simeq H^q(X^\an,\Os_{X^\an}(n))$ pour tout $q\in\N$ si, et seulement si, on a un isomorphisme $H^q(X,\Os_{X}(n-1)) \simeq H^q(X^\an,\Os_{X^\an}(n-1))$ pour tout $q\in\N$. 

Un calcul explicite montre que l'on a  $H^q(X,\Os_{X}) \simeq H^q(X^\an,\Os_{X^\an})$ pour tout $q\in\N$. On en d\'eduit le r\'esultat annonc\'e.

\subsubsection{Assertion~{\it i)} en g\'en\'eral}

Soit~$r\in\N$. Posons $X=\P^r_{\As}$. Nous allons d\'emontrer, par une r\'ecurrence descendante sur~$q$, que, pour tout entier naturel~$q$, l'assertion~{\it i)} du th\'eor\`eme est vraie pour $H^q$.

Si $q>r$, pour tout faisceau de $\Os_{X}$-modules coh\'erent~$\Fs$, les groupes $H^q(X,\Fs)$ et $H^q(X^\an,\Fs^\an)$ sont tous deux nuls, et le r\'esultat est vrai. 

Soit $q\in\N^*$ tel que le r\'esultat soit vrai pour $H^q$. Soit~$\Fs$ un faisceau de $\Os_{X}$-modules coh\'erent. Nous pouvons l'ins\'erer dans une suite exacte de faisceaux de $\Os_{X}$-modules coh\'erents
$$0 \to \Rs \to \Ls \to \Fs \to 0,$$
o\`u~$\Ls$ est somme directe de faisceaux isomorphes \`a~$\Os(n)$, avec $n\in\Z$. 

Ins\'erons le faisceau~$\Rs$ dans une suite exacte courte du m\^eme type 
$$0 \to \Rs' \to \Ls' \to \Rs \to 0.$$
D'apr\`es l'hypoth\`ese de r\'ecurrence et le num\'ero~\ref{iOn}, le r\'esultat vaut pour le faisceau~$\Rs'$ en rang~$q$ et pour le faisceau~$\Ls'$ en rangs~$q$ et~$q-1$. Le lemme des cinq assure alors que le morphisme
$$H^{q-1}(X,\Rs) \to H^{q-1}(X^\an,\Rs^\an)$$
est surjectif.
 
Pour les m\^emes raisons que pr\'ec\'edemment, nous savons en outre que le r\'esultat vaut pour le faisceau~$\Rs$ en rang~$q$ et pour le faisceau~$\Ls$ en rangs~$q$ et~$q-1$. Une nouvelle application du lemme des cinq assure alors que le morphisme
$$H^{q-1}(X,\Fs) \to H^{q-1}(X^\an,\Fs^\an)$$
est bijectif.

\subsubsection{Pleine fid\'elit\'e du foncteur $\Fs \to \Fs^\an$}

Soit~$X$ un espace projectif sur~$\As$. Soient~$\Fs$ et~$\Gs$ deux faisceaux de $\Os_{X}$-modules coh\'erents. Soit~$x^\an$ un point de~$X^\an$. Notons~$x$ son image dans~$X$.

Nous disposons des isomorphismes
$$\Hom(\Fs,\Gs)^\an_{x^\an} \simeq \textrm{Hom}(\Fs,\Gs)_{x} \otimes_{\Os_{X,x}} \Os_{X^\an,x^\an}$$
et
$$\Hom(\Fs^\an,\Gs^\an)_{x^\an} \simeq \textrm{Hom}(\Fs_{x} \otimes_{\Os_{X,x}} \Os_{X^\an,x^\an}, \Gs_{x} \otimes_{\Os_{X,x}} \Os_{X^\an,x^\an}).$$

La platitude du morphisme naturel $X^\an \to X$ entra\^{\i}ne que le morphisme
$$\Hom(\Fs,\Gs)^\an \to \Hom(\Fs^\an,\Gs^\an)$$
est un isomorphisme.

On conclut en appliquant le r\'esultat de l'assertion~{\it i)} du th\'eor\`eme au faisceau coh\'erent $\Hom(\Fs,\Gs)$ et \`a l'entier $q=0$.

\subsubsection{Surjectivit\'e essentielle du foncteur $\Fs \to \Fs^\an$}

Nous allons d\'emontrer, par r\'ecurrence sur~$r$, que, pour tout entier naturel~$r$, le foncteur $\Fs \to \Fs^\an$ est essentiellement surjectif lorsque $X = \P^r_{\As}$.

Lorsque $r=0$, le r\'esultat est classique (\emph{cf.}~\cite{rouge}, proposition~2.3.1).

Soit~$r\in\N$ tel que le r\'esultat soit vrai pour le sch\'ema $\P^r_{\As}$. Posons $X=\P^{r+1}_{\As}$. Commen\c{c}ons par une s\'erie de lemmes.

\begin{lem}\label{lemhyp}
Pour tout hyperplan projectif~$Y$ de~$X$ et tout faisceau de $\Os_{X^\an}$-modules coh\'erent~$\Ns$, il existe un entier~$n_{0}$ tel que
$$\forall n\ge n_{0}, \forall q\ge 1,\ H^q(Y^\an,\Ns_{|Y^\an}(n))=0.$$
\end{lem}
\begin{proof}
On d\'emontre ce r\'esultat en appliquant l'hypoth\`ese de r\'ecurrence au faisceau~$\Ns_{|Y^\an}$, puis en utilisant le r\'esultat analogue pour les sch\'emas projectifs et les isomorphismes fournis par l'assertion~{\it i)} du th\'eor\`eme.
\end{proof}

\begin{lem}\label{chgtbasekL}
Soit~$L$ une extension valu\'ee compl\`ete de~$k$. Notons $\pi : X^\an_{L} \to X^\an$ le morphisme de projection. Soient~$x$ un point de~$X^\an$ et~$x_{L}$ l'un de ses ant\'ec\'edents par le morphisme~$\pi$.  Soit~$\Fs$ un faisceau de $\Os_{X^\an}$-modules coh\'erent. Supposons que la fibre~$\pi^*(\Fs)_{x_{L}}$ soit engendr\'ee par l'ensemble des sections globales de~$\pi^*(\Fs)$. Alors, la fibre~$\Fs_{x}$ est engendr\'ee par l'ensemble des sections globales de~$\Fs$.
\end{lem}
\begin{proof}
D'apr\`es le th\'eor\`eme~\ref{kiehlfini}, $\Fs(X^\an)$ est un $\As$-module de Banach fini. Consid\'erons une famille $(f_{1},\ldots,f_{r})$, avec $r\in\N$, qui l'engendre. Notons~$\Gs$ le conoyau du morphisme $\Os_{X^\an}^r \to \Fs$ d\'efini par cette famille. Puisque le produit tensoriel est exact \`a droite, le faisceau~$\pi^*(\Gs)$ est le conoyau du morphisme $\Os_{X^\an_{L}}^r \to \pi^*(\Fs)$ d\'efini par la famille $(\pi^*(f_{1}),\ldots, \pi^*(f_{r}))$. 

L'exactitude du foncteur $\cdot \hat{\otimes}_{k} L$ assure que les $(\As\hat{\otimes}_{k} L)$-modules $\pi^*(\Fs)(X^\an_{L})$ et $\Fs(X^\an) \hat{\otimes}_{k} L$ sont isomorphes. En particulier, la famille $(\pi^*(f_{1}),\ldots, \pi^*(f_{r}))$ engendre $\pi^*(\Fs)(X^\an_{L})$. Puisque, par hypoth\`ese, cet ensemble engendre $\pi^*(\Fs)_{x_{L}}$, la fibre $\pi^*(\Gs)_{x_{L}} \simeq \Gs_{x} \otimes_{\Os_{X^\an,x}} \Os_{X^\an_{L},x_{L}}$ est nulle. {\it A fortiori}, nous avons $\Gs_{x} \otimes_{\Os_{X^\an,x}} \kappa(x_{L})=0$. Puisque $\Gs_{x} \otimes_{\Os_{X^\an,x}} \kappa(x_{L}) \simeq \Gs_{x} \otimes_{\Os_{X^\an,x}} \kappa(x) \otimes_{\kappa(x)} \kappa(x_{L})$, nous avons m\^eme $\Gs_{x} \otimes_{\Os_{X^\an,x}} \kappa(x)=0$, d'o\`u l'on d\'eduit que $\Gs_{x}=0$, par le lemme de Nakayama.
\end{proof}

\begin{lem}\label{lemgen}
Soient~$\Ns$ un faisceau de $\Os_{X^\an}$-modules coh\'erent et~$x$ un point de~$X^\an$. Il existe un entier naturel~$n_{0}$ tel que, pour tout $n\ge n_{0}$, la fibre~$\Ns(n)_{x}$ soit engendr\'ee par l'ensemble des sections globales de~$\Ns(n)$.
\end{lem}
\begin{proof}
D'apr\`es le lemme~\ref{chgtbasekL}, quitte \`a effectuer un changement de corps de base de~$k$ \`a~$\Hs(x)$ (et \`a modifier les autres donn\'ees en cons\'equence), nous pouvons supposer que le point~$x$ est $k$-rationnel. Il est alors situ\'e sur l'analytifi\'e~$Y^\an$ d'un certain hyperplan projectif~$Y$ de~$X$.

Soit~$t$ une section de~$\Os_{X}(1)$ de lieu des z\'eros~$Y$. La multiplication par~$t$ d\'efinit une suite exacte
$$0 \to \Ns' \to \Ns(-1) \to \Ns \to \Ns_{|Y^\an} \to 0,$$
o\`u~$\Ns'$ est un faisceau de $\Os_{X^\an}$-modules coh\'erent support\'e par~$Y^\an$.

Soit~$n\in\Z$. En tensorisant la suite pr\'ec\'edente par~$\Os_{X^\an}(n)$ puis en la scindant, nous obtenons les deux suites exactes courtes
$$0 \to \Ns'(n) \to \Ns(n-1) \to \Ps_{n} \to 0$$
et
$$0 \to \Ps_{n} \to \Ns(n) \to \Ns_{|Y^\an}(n) \to 0,$$
qui donnent naissance aux deux suites exactes de cohomologie
$$H^1(X^\an,\Ns(n-1)) \to H^1(X^\an,\Ps(n)) \to H^2(X^\an,\Ns'(n))$$
et
$$H^1(X^\an,\Ps_{n}) \to H^1(X^\an,\Ns(n)) \to H^1(X^\an,\Ns_{|Y^\an}(n)).$$

D'apr\`es le lemme~\ref{lemhyp}, il existe un entier~$n_{1}$ tel que, pour tout $n\ge n_{1}$, les groupes de cohomologie
$$H^2(X^\an,\Ns'(n)) \textrm{ et } H^1(X^\an,\Ns_{|Y^\an}(n))$$
soient nuls et, par cons\'equent, le morphisme compos\'e 
$$H^1(X^\an,\Ns(n-1)) \to H^1(X^\an,\Ps_{n}) \to H^1(X^\an,\Ns(n))$$
soit surjectif.

Le morphisme $X^\an \to \Mc(\As)$ \'etant propre, le th\'eor\`eme~\ref{kiehlfini}, assure que le $\As$-module $H^1(X^\an,\Ns(n_{1}-1))$ est de type fini, et donc noeth\'erien. On en d\'eduit qu'il existe un entier $n_{2}\ge n_{1}$ tel que, pour tout $n\ge n_{2}$, le morphisme
$$H^1(X^\an,\Ns(n-1)) \to H^1(X^\an,\Ns(n))$$
soit un isomorphisme. Par cons\'equent, pour tout $n\ge n_{2}$, le morphisme surjectif 
$$H^1(X^\an,\Ps_{n}) \to H^1(X^\an,\Ns(n))$$ 
est bijectif, d'o\`u l'on d\'eduit, en consid\'erant la suite exacte longue associ\'e \`a la seconde suite exacte courte, que le morphisme
$$H^0(X^\an,\Ns(n)) \to H^0(X^\an,\Ns_{|Y^\an}(n))$$
est surjectif.

D'apr\`es l'hypoth\`ese de r\'ecurrence, le faisceau de $\Os_{Y^\an}$-modules coh\'erent~$\Ns_{|Y^\an}$ est l'analytifi\'e d'un faisceau de $\Os_{Y}$-modules coh\'erent~$\Gs$. Notons~$x^\alg$ l'image du point~$x$ dans~$Y$. Les r\'esultats classiques sur les sch\'emas projectifs assurent qu'il existe un entier $n_{0} \ge n_{2}$ tel que, pour tout $n\ge n_{0}$, la fibre $\Gs(n)_{x^\alg}$ soit engendr\'ee, en tant que $\Os_{Y,x^\alg}$-module, par l'ensemble des sections globales $H^0(Y,\Gs(n))$. En utilisant l'assertion~{\it i)} du th\'eor\`eme, on en d\'eduit que le r\'esultat vaut encore en rempla\c{c}ant respectivement~$Y$ par~$Y^\an$, $\Gs$ par~$\Ns_{|Y^\an}$ et~$x^\alg$ par~$x$.

Notons~$\Is$ le faisceau d'id\'eaux qui d\'efinit~$Y^\an$ dans~$X^\an$. Remarquons que sa fibre~$\Is_{x}$ est contenue dans l'id\'eal maximal~$\m_{x}$ de~$\Os_{X^\an,x}$. Soit~$n\ge n_{0}$. Nous venons de montrer que $\Ns(n)_{x}\otimes_{\Os_{X^\an,x}} (\Os_{X^\an,x}/\Is_{x})$ est engendr\'e par $H^0(X^\an,\Ns_{|Y^\an}(n))$. On en d\'eduit que $\Ns(n)_{x}\otimes_{\Os_{X^\an,x}} (\Os_{X^\an,x}/\m_{x})$ est engendr\'e par $H^0(X^\an,\Ns_{|Y^\an}(n))$, et donc par $H^0(X^\an,\Ns(n))$. On conclut par le lemme de Nakayama.
\end{proof}

Terminons, \`a pr\'esent, la d\'emonstration. Soit~$\Ns$ un faisceau de $\Os_{X^\an}$-modules coh\'erent. En utilisant le r\'esultat du lemme pr\'ec\'edent et la compacit\'e de~$X^\an$, on montre qu'il existe un entier~$n$ tel qu'en tout point~$x$ de~$X^\an$, la fibre~$\Ns(n)_{x}$ soit engendr\'ee par $H^0(X^\an,\Ns(n))$. On en d\'eduit l'existence d'un entier naturel~$p$, d'un faisceau de $\Os_{X^\an}$-modules coh\'erent~$\Rs$ et d'une suite exacte
$$0 \to \Rs \to \Os_{X^\an}(-n)^p \to \Ns \to 0.$$
En appliquant le m\^eme raisonnement au faisceau~$\Rs$, nous parvenons finalement \`a \'ecrire le faisceau~$\Ns$ comme le conoyau d'un morphisme $\varphi : \Os_{X^\an}(-m)^q \to \Os_{X^\an}(-n)^p$, avec $m\in\Z$ et $q\in\N$. Puisque le foncteur d'analytification est pleinement fid\`ele, le morphisme~$\varphi$ est l'analytifi\'e d'un morphisme $\varphi^\alg : \Os_{X}(-m)^q \to \Os_{X}(-n)^p$. L'exactitude \`a droite du foncteur d'analytification assure alors que le faisceau~$\Ns$ est isomorphe \`a l'analytifi\'e du conoyau du morphisme~$\varphi^\alg$, qui est un faisceau de $\Os_{X}$-modules coh\'erent.

\subsection{Corollaires}

Nous \'enon\c{c}ons ici deux corollaires du th\'eor\`eme~\ref{GAGAsuraffinoide}. Ils ont \'egalement pour objet des r\'esultats de type GAGA, mais sur des bases qui ne sont plus n\'ecessairement affino\"{\i}des. Nous sommes convaincu qu'il est possible de les \'etendre \`a une base quelconque, de fa\c{c}on \`a obtenir un analogue parfait des th\'eor\`emes obtenus par M.~Hakim dans le cadre de la g\'eom\'etrie analytique complexe (\emph{cf.}~\cite{Hakim}, chapitre~VIII, th\'eor\`emes~3.2 et~3.5). Cependant, pour \'eviter d'avoir \`a utiliser le formalisme un peu lourd des sch\'emas relatifs sur un espace analytique, nous nous contenterons d'\'enoncer les deux cas particuliers que nous utilisons dans cet article.\\

Soit~$k$ un corps muni d'une valeur absolue ultram\'etrique pour laquelle il est complet, $\As$ une alg\`ebre $k$-affino\"{\i}de et~$X$ un $\As$-sch\'ema propre. Soit~$B$ une partie compacte de~$\Mc(\As)$ poss\'edant un syst\`eme fondamental de voisinages affino\"{\i}des. Rappelons que la notation~$\Os(B)$ d\'esigne l'anneau des germes de fonctions analytiques au voisinage de~$B$. Notons $Y = X\times_{\Spec(\As)}\Spec(\Os(B))$ et~$Y^\an$ l'image r\'eciproque de~$B$ dans~$X^\an$. Munissons~$Y^\an$ du faisceau des fonctions surconvergentes. En utilisant le morphisme d'analytification au-dessus d'un espace affino\"{\i}de d\'efini par V.~Berkovich, on montre qu'il existe un morphisme d'espaces localement annel\'es $Y^\an \to Y$. La r\'etrotirette d'un faisceau de $\Os_{Y}$-modules coh\'erent~$\Fs$ est un faisceau de $\Os_{Y^\an}$-modules coh\'erent, que nous noterons~$\Fs^\an$.

\begin{cor}\label{GAGAvoiscompact}
Supposons que nous nous trouvons dans la situation d\'ecrite ci-dessus. Alors
\begin{enumerate}[\it i)]
\item pour tout faisceau de $\Os_{Y}$-modules coh\'erent~$\Fs$ et tout entier $q\in\N$, le morphisme
$$H^q(Y,\Fs) \to H^q(Y^\an,\Fs^\an)$$
est un isomorphisme ;
\item le foncteur d'analytification
$$\Fs \to \Fs^\an$$
induit une \'equivalence entre la cat\'egorie des $\Os_{Y}$-modules coh\'erents et celle des $\Os_{Y^\an}$-modules coh\'erents.
\end{enumerate}
\end{cor}
\begin{proof}
Il faut tout d'abord remarquer que l'espace~$Y^\an$ est compact. En reprenant le raisonnement de la preuve de la proposition~$1$ de~\cite{SemCartan419}, on en d\'eduit que tout faisceau coh\'erent sur~$Y^\an$ se prolonge en un faisceau coh\'erent sur un voisinage de~$Y^\an$, et donc sur une partie de la forme~$Y^\an\times_{\Mc(\As)} V$, o\`u~$V$ est un voisinage affino\"{\i}de de~$B$. En utilisant ce raisonnement et le th\'eor\`eme~\ref{GAGAsuraffinoide}, on obtient le r\'esultat attendu.
\end{proof}

Soit~$k$ un corps muni d'une valeur absolue ultram\'etrique pour laquelle il est complet, $\As$ une alg\`ebre $k$-affino\"{\i}de et~$X$ un $\As$-sch\'ema propre. Soit~$B$ un espace $\As$-analytique qui soit limite inductive d'espaces affino\"{\i}des. Notons $Z = X\times_{\Spec(\As)}\Spec(\Os(B))$. En utilisant le morphisme d'analytification au-dessus d'un espace affino\"{\i}de d\'efini par V.~Berkovich, on construit, par limite inductive, un espace analytique~$Z^\an$ et un morphisme d'espaces localement annel\'es $Z^\an \to Z$. Comme pr\'ec\'edemment, la r\'etrotirette d'un faisceau de $\Os_{Z}$-modules coh\'erent~$\Fs$ est un faisceau de $\Os_{Z^\an}$-modules coh\'erent, que nous noterons~$\Fs^\an$. Le r\'esultat suivant se d\'eduit alors ais\'ement du th\'eor\`eme~\ref{GAGAsuraffinoide}.

\begin{cor}\label{GAGAsurouvert}
Supposons que nous nous trouvons dans la situation d\'ecrite ci-dessus. Alors
\begin{enumerate}[\it i)]
\item pour tout faisceau de $\Os_{Z}$-modules coh\'erent~$\Fs$ et tout entier $q\in\N$, le morphisme
$$H^q(Z,\Fs) \to H^q(Z^\an,\Fs^\an)$$
est un isomorphisme ;
\item le foncteur d'analytification
$$\Fs \to \Fs^\an$$
induit une \'equivalence entre la cat\'egorie des $\Os_{Z}$-modules coh\'erents et celle des $\Os_{Z^\an}$-modules coh\'erents.
\end{enumerate}
\end{cor}

\section{La droite de Berkovich sur un anneau d'entiers de corps de nombres}\label{sectionberko}


Dans cette annexe, nous pr\'esentons succintement la droite de Berkovich sur un anneau d'entiers de corps de nombres. Nous invitons le lecteur dont ces pr\'emices auront \'eveill\'e la curiosit\'e \`a parcourir l'ouvrage~\cite{monjolimemoire} pour approfondir ce sujet.

\subsection{D\'efinitions}

Soit~$K$ un corps de nombres. Notons~$A$ l'anneau de ses entiers. Commen\c{c}ons par rappeler la d\'efinition d'espace affine analytique sur~$A$. Elle est due \`a V.~Berkovich (\emph{cf.}~\cite{rouge}, \S 1.5). Soit $n\in\N$. L'espace affine analytique de dimension~$n$ sur~$A$, not\'e~$\E{n}{A}$, est l'ensemble des semi-normes multiplicatives sur $A[T_{1},\ldots,T_{n}]$, c'est-\`a-dire l'ensemble des applications 
$$|.| : A[T_{1},\ldots,T_{n}] \to \R_{+}$$ 
qui v\'erifient les propri\'et\'es suivantes :
\begin{enumerate}[\it i)]
\item $|0|=0$ et $|1|=1$ ;
\item $\forall P,Q\in A[T_{1},\ldots,T_{n}],\, |P+Q|\le |P|+|Q|$ ;
\item $\forall P,Q\in A[T_{1},\ldots,T_{n}], \,|PQ|= |P|\, |Q|$.
\end{enumerate}

\begin{rem}
Dans la d\'efinition propos\'ee par V.~Berkovich figure une condition suppl\'ementaire qui fait intervenir une norme sur l'anneau~$A$. Pour $a\in A$, posons 
$$\|a\| = \max_{\sigma\in \Sigma_{\infty}}(|\sigma(a)|_{\infty}),$$ 
o\`u~$\Sigma_{\infty}$ d\'esigne l'ensemble des plongements complexes du corps~$K$ et~$|.|_{\infty}$ la valeur absolue usuelle sur~$\C$. La fonction $\l.\l : A\to \R_{+}$ d\'efinit une norme sur~$A$ et, lorsque l'on munit l'anneau~$A$ de cette norme, la d\'efinition de Berkovich co\"incide avec la n\^otre. Signalons que, quelle que soit la norme dont on munit~$A$ (sous r\'eserve tout de m\^eme qu'elle soit sous-multiplicative et fasse de~$A$ un espace complet), on obtient un espace contenu dans celui que nous avons d\'efini.
\end{rem}

Soit~$x$ un point de~$\E{n}{A}$. Il lui est associ\'e une semi-norme multiplicative~$|.|_{x}$ sur $A[T_{1},\ldots,T_{n}]$. L'ensemble~$\p_{x}$ des \'el\'ements sur lesquels elle s'annule est un id\'eal premier de $A[T_{1},\ldots,T_{n}]$. Le quotient est un anneau int\`egre sur lequel la semi-norme~$|.|_{x}$ induit une valeur absolue. Nous noterons~$\Hs(x)$ le compl\'et\'e du corps des fractions de cet anneau pour cette valeur absolue. Nous noterons simplement~$|.|$ la valeur absolue sur le corps~$\Hs(x)$, cela n'entra\^{\i}nant pas de confusion. La construction fournit un morphisme 
$$A[T_{1},\ldots,T_{n}] \to \Hs(x).$$
L'image d'un \'el\'ement~$P$ de $A[T_{1},\ldots,T_{n}]$ par ce morphisme sera not\'ee~$P(x)$. Avec ces notations, nous avons donc $|P(x)|=|P|_{x}$.

Munissons, \`a pr\'esent, l'espace analytique $\E{n}{A}$ d'une topologie : celle engendr\'ee par les ensembles de la forme
$$\{x \in \E{n}{A}\, |\, r < |P(x)| < s\},$$ 
pour $P \in A[T_{1},\ldots,T_{n}]$ et $r,s\in\R$.

Pour finir, nous d\'efinissons un faisceau d'anneaux~$\Os$ sur~$\E{n}{A}$ de la fa\c{c}on suivante : pour tout ouvert~$U$ de~$\E{n}{A}$, l'anneau~$\Os(U)$ est constitu\'e des applications
$$f : U\to \bigsqcup_{x\in U} \Hs(x)$$
qui v\'erifient les deux conditions suivantes :
\begin{enumerate}[\it i)]
\item $\forall x\in U$, $f(x)\in\Hs(x)$ ;
\item $f$ est localement limite uniforme de fractions rationnelles sans p\^oles.
\end{enumerate}

\subsection{Dimension~$0$}

Afin de rendre plus palpables les d\'efinitions pr\'ec\'edentes, nous allons d\'ecrire explicitement~$\E{0}{A}$, l'espace affine analytique de dimension~$0$ sur~$A$, que nous noterons plus volontiers~$\Mc(A)$. Nous noterons~$|.|_{\infty}$ la valeur absolue usuelle sur~$\C$ et, pour tout id\'eal maximal~$\m$ de~$A$, nous noterons~$|.|_{\m}$ la valeur absolue $\m$-adique normalis\'ee. Du th\'eor\`eme d'Ostrowski, l'on d\'eduit que les points de~$\Mc(A)$ sont exactement 
\begin{enumerate}[\it i)]
\item la valeur absolue triviale~$|.|_{0}$ (nous noterons~$a_{0}$ le point associ\'e) ;
\item la valeur absolue archim\'edienne $|\sigma(.)|_{\infty}^\eps$ (nous noterons~$a_{\sigma}^\eps$ le point associ\'e) pour tout toute classe de conjugaison de plongements complexes~$\sigma$ de~$K$ et tout \'el\'ement~$\eps$ de $\of{]}{0,1}{]}$  ;
\item la valeur absolue $\m$-adique $|.|_{\m}^\eps$ (nous noterons~$a_{\m}^\eps$ le point associ\'e) pour tout id\'eal maximal~$\m$ de~$A$ et tout \'el\'ement~$\eps$ de $\of{]}{0,+\infty}{[}$ ;
\item la semi-norme $|.|_{\m}^{+\infty}$ (nous noterons~$\tilde{a}_{\m}$ le point associ\'e) induite par la valeur absolue triviale sur le corps fini $A/\m$ pour tout id\'eal maximal~$\m$ de~$A$.
\end{enumerate}

\begin{figure}[htb]
\begin{center}
\input{corpsres.pstex_t}
\caption{L'espace $\Mc(\Z)$.}
\end{center}
\end{figure}

Nous pouvons \'egalement d\'ecrire la topologie de l'espace~$\Mc(A)$ (\emph{cf.} figure~\thefigure, trac\'ee dans le cas o\`u $K=\Q$ pour simplifier les notations, mais ais\'ement g\'en\'eralisable). Pour cela, il suffit d'indiquer que chacune des branches trac\'ee sur la figure~\thefigure\ est hom\'eomorphe \`a un segment r\'eel et qu'un voisinage du point central~$a_{0}$ est une partie qui contient enti\`erement toutes les branches \`a l'exception d'un nombre fini, et qui contient un voisinage de~$a_{0}$ dans chacune des branches restantes. Si l'on pr\'ef\`ere, l'espace~$\Mc(A)$ poss\`ede la topologie du compactifi\'e d'Alexandrov de la r\'eunion disjointe de ses branches priv\'ees de~$a_{0}$, le point~$a_{0}$ jouant le r\^ole du point \`a l'infini.


Nous pouvons \'egalement d\'ecrire explicitement les sections du faisceau structural sur les ouverts de~$\Mc(A)$. Nous avons repr\'esent\'e les diff\'erents cas \`a la figure~\addtocounter{figure}{1}\thefigure\addtocounter{figure}{-1}, de nouveau dans le cas o\`u $K=\Q$.

\begin{figure}[htb]
\begin{center}
\input{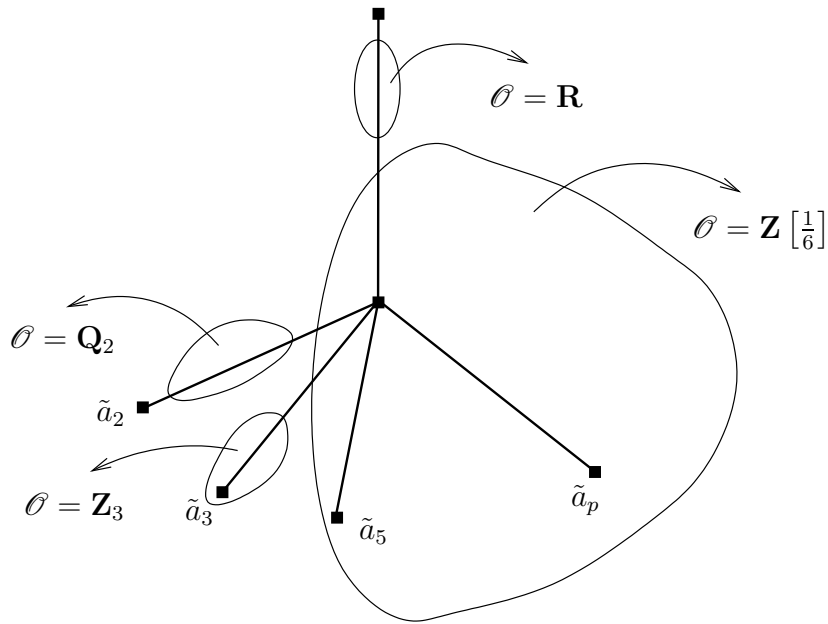}
\caption{Le faisceau structural sur $\Mc(\Z)$.}
\end{center}
\end{figure}

\subsection{Dimension~$1$}

Venons-en, \`a pr\'esent, \`a l'espace affine analytique de dimension~$1$ sur~$A$. Nous noterons~$T$ la coordonn\'ee sur cet espace. Remarquons, tout d'abord, que le morphisme $A \to A[T]$ induit un morphisme de projection
$$\pi : \E{1}{A} \to \Mc(A).$$
Cela permet d'obtenir une description topologique de la droite de Berkovich sur~$A$ : la fibre de~$\pi$ au-dessus d'un point~$x$ de~$\Mc(A)$ est isomorphe \`a la droite de Berkovich sur le corps $\Hs(x)$. Si $\Hs(x)=\C$, cette droite est isomorphe \`a l'espace~$\C$ et, si $\Hs(x)=\R$, elle est isomorphe \`a son quotient par la conjugaison complexe. Nous ne chercherons pas \`a obtenir de description plus pr\'ecise et nous contenterons d'indiquer quelques propri\'et\'es (\emph{cf.}~\cite{monjolimemoire}, th\'eor\`emes~4.4.1 et~4.5.5).


\begin{thm}\label{resume}
\begin{enumerate}[\it i)]
\item L'espace $\E{1}{A}$ est localement compact, m\'etrisable et de dimension topologique~$3$.
\item L'espace $\E{1}{A}$ est localement connexe par arcs.
\item Le morphisme de projection $\pi : \E{1}{A}\to \Mc(A)$ est ouvert.
\item En tout point~$x$ de~$\E{1}{A}$, l'anneau local~$\Os_{x}$ est hens\'elien, noeth\'erien, r\'egulier, de dimension inf\'erieure \`a $2$ et le corps r\'esiduel~$\kappa(x)$ est hens\'elien.
\item Le faisceau structural~$\Os$ est coh\'erent.
\end{enumerate}
\end{thm}

Signalons encore que la droite de Berkovich sur~$A$ satisfait au principe du prolongement analytique (\emph{ibid.}, th\'eor\`emes~4.4.2 et~7.1.9, corollaire~4.4.5).

\begin{thm}\label{prolan}
Soit~$U$ une partie connexe de~$\E{1}{A}$.
\begin{enumerate}[\it i)]
\item Le principe du prolongement analytique vaut sur~$U$. En particulier, l'anneau~$\Os(U)$ est int\`egre.
\item L'anneau des sections m\'eromorphes~$\Ms(U)$ est un corps.
\item Si~$U$ est de Stein, le morphisme naturel $\Frac(\Os(U))\to \Ms(U)$ est un isomorphisme.
\end{enumerate}
\end{thm}

Rappelons ici ce que nous entendons par espace de Stein. Nous dirons qu'un espace localement annel\'e $(X,\Os_{X})$ est de Stein s'il satisfait les conclusions des th\'eor\`emes~A et~B de H.~Cartan : 
\begin{enumerate}[A)]
\item pour tout faisceau de $\Os_{X}$-modules coh\'erent~$\Fs$ et tout point~$x$ de~$X$,
la fibre~$\Fs_{x}$ est engendr\'ee par l'ensemble des sections globales~$\Fs(X)$ ;
\item pour tout faisceau de $\Os_{X}$-modules coh\'erent~$\Fs$ et tout entier $q\in\N^*$,
nous avons $H^q(X,\Fs)=0$.
\end{enumerate}

Donnons quelques exemples de sous-espaces de la droite analytique~$\AZ$ qui sont des espaces de Stein (\emph{ibid.}, th\'eor\`eme~6.6.29).

\begin{thm}\label{lemniscateStein}
Soient~$V$ une partie ouverte et connexe de l'espace~$\Mc(A)$ et \mbox{$s,t\in\R$}. Soit~$P(T)$ un polyn\^ome unitaire \`a coefficients dans~$\Os(V)$. Les parties suivantes de la droite analytique~$\E{1}{A}$ sont des espaces de Stein :
\begin{enumerate}[\it i)]
\item $\left\{x\in \pi^{-1}(V)\, \big|\, s < |P(T)(x)| < t\right\}$ ;
\item $\left\{x\in \pi^{-1}(V)\, \big|\, |P(T)(x)| > s\right\}$.
\end{enumerate}
\end{thm}

\bigskip

Pour terminer, disons quelques mots des sections globales sur les parties de la droite analytique~$\E{1}{A}$. Sur les disques, elles s'expriment essentiellement en termes de s\'eries dont les coefficients sont des fonctions sur~$\Mc(A)$. Consid\'erons, par exemple, le disque ouvert relatif de rayon~$1$ :
$$\D=\left\{\left. x\in\E{1}{A}\, \right|\, |T(x)|<1\right\}.$$
Le morphisme naturel $A[T]\to \Os(\D)$ induit un isomorphisme
$$A_{1^-}[\![T]\!] \xrightarrow[]{\sim} \Os(\D),$$
o\`u~$A_{1^-}[\![T]\!]$ d\'esigne l'anneau constitu\'e des s\'eries de la forme
$$\sum_{n\ge 0} a_{n}\, T^n \in A[\![T]\!]$$
telles que le rayon de convergence de la s\'erie \`a coefficients complexes
$$\sum_{n\ge 0} \sigma(a_{n})\, T^n$$
soit sup\'erieur ou \'egal \`a~$1$, pour tout plongement complexe~$\sigma$ de~$K$. On d\'eduit cette description du th\'eor\`eme~3.2.16 de \emph{ibid.} 

\`A partir de la description des anneaux de sections sur les disques, nous pouvons d\'eduire celle des anneaux locaux en certains points. Nous nous contenterons de deux exemples. Soit~$\m$ un id\'eal maximal de~$A$. Notons~$z_{\m}$ le point~$0$ de la fibre de~$\pi$ au-dessus du point~$\tilde{a}_{\m}$. D'apr\`es le corollaire~3.2.5 de \emph{ibid.}, le morphisme naturel $A[T] \to \Os_{z_{p}}$ induit un isomorphisme
$$\hat{A}_{\m}[\![T]\!] \xrightarrow[]{\sim} \Os_{z_{p}}.$$
Notons~$z_{0}$ le point~$0$ de la fibre de~$\pi$ au-dessus du point~$a_{0}$. D'apr\`es le corollaire~3.2.8 de \emph{ibid.}, le morphisme naturel $A[T] \to \Os_{z_{0}}$ induit un isomorphisme
$$E \xrightarrow[]{\sim} \Os_{z_{0}},$$
o\`u~$E$ d\'esigne l'anneau constitu\'e des \'el\'ements~$f$ de~$K[\![T]\!]$ qui v\'erifient les propri\'et\'es suivantes :
\begin{enumerate}[\it i)]
\item $\exists a\in A\setminus\{0\},\ f(aT)\in A[\![T]\!]$ ;
\item pour tout plongement complexe~$\sigma$ de~$K$, le rayon de convergence complexe de la s\'erie~$\sigma(f)$ est strictement positif ;
\item pour tout id\'eal maximal~$\m$ de~$A$, le rayon de convergence $\m$-adique de la s\'erie~$f$ est strictement positif (il suffit d'imposer cette condition pour les id\'eaux maximaux qui contiennent un \'el\'ement~$a$ poss\'edant les propri\'et\'es d\'ecrites en~{\it i}).
\end{enumerate}

\nocite{}
\bibliographystyle{plain}
\bibliography{biblio}


\end{document}